\documentclass[final,leqno]{siamltex}
\usepackage[leqno]{amsmath}
\usepackage{graphicx}% Figure
\usepackage{graphics}
\usepackage{epstopdf}
\usepackage{subfigure}% Subfigure
\usepackage{threeparttable}

\usepackage{bm}
\usepackage{amssymb}
\usepackage{leftidx}
\usepackage{empheq}
\usepackage{color}
\usepackage{makecell}
\usepackage{enumerate}
\allowdisplaybreaks
\renewcommand{\theequation}{\arabic{section}.\arabic{equation}}
\numberwithin{equation}{section}
\newtheorem{remark}{Remark}[section]
\title{A highly efficient and accurate exponential semi-implicit scalar auxiliary variable (ESI-SAV) approach for dissipative system.
        \thanks{
We would like to acknowledge the assistance of volunteers in putting together this example manuscript and supplement. This work is supported by China Postdoctoral Science Foundation under grant numbers BX20190187, 2019M650152 and 2020M672111, by National Natural Science Foundation of China (Grant Nos: 11931003, 41974133, 11901489, 11971276).}}

      \author{Zhengguang Liu
             \thanks{School of Mathematics and Statistics, Shandong Normal University, Jinan, China. Email: liuzhgsdu@yahoo.com}.
                                       \and
             Xiaoli Li\textsuperscript{*}
             \thanks{Corresponding author: Fujian Provincial Key Laboratory on Mathematical Modeling and High Performance Scientific Computing and School of Mathematical Sciences, Xiamen University, Xiamen, Fujian, 361005, China. Email: xiaolisdu@163.com}. }
\begin{document}
\UseRawInputEncoding
\maketitle

\begin{abstract}
The scalar auxiliary variable (SAV) approach \cite{ShenA} is a very popular and efficient method to simulate various phase field models. To save the computational cost, a new SAV approach is given in \cite{huang2020highly} by introducing a new variable $\theta$. The new SAV approach can be proved to save nearly half CPU time of the original SAV approach while keeping all its other advantages. In this paper, we propose a novel technique to construct an exponential semi-implicit scalar auxiliary variable (ESI-SAV) approach without introducing any extra variables. The new proposed method also only needs to solve one linear equation with constant coefficients at each time step. Furthermore, the constructed ESI-SAV method does not need the bounded below restriction of nonlinear free energy potential which is more reasonable and effective for various phase field models. Meanwhile it is easy to construct first-order, second-order and higher-order unconditionally energy stable time-stepping schemes. Other than that, the ESI-SAV approach can be proved to be effective to solve the non-gradient but dissipative system such as Navier-Stokes equations. Several numerical examples are provided to demonstrate the improved efficiency and accuracy of the proposed method.
\end{abstract}



\pagestyle{myheadings}
\thispagestyle{plain}
\markboth{ZHENGGUANG LIU AND XIAOLI LI} {ESI-SAV APPROACH FOR DISSIPATIVE SYSTEM}
   %==================================================================
  \section{Introduction}
The phase field models are very important and popular dissipative systems which cover a lot of fields such as alloy casting, new material preparation, image processing, finance and so on \cite{ambati2015review,guo2015thermodynamically,liu2019efficient,liu2020two,marth2016margination,miehe2010phase,shen2015efficient,wheeler1993computation}. Many classical phase field models such as Allen-Cahn model \cite{ainsworth2017analysis,du2019maximum,guan2014second,shen2010numerical,yang2020convergence,zhai2014numerical}, Cahn-Hilliard model \cite{chen2019fast,du2018stabilized,he2007large,shen2010numerical,weng2017fourier,yang2018numerical,zhu1999coarsening} and phase field crystal model \cite{li2018unconditionally,li2019efficient,li2017efficient,liu2020two,yang2017linearly} have been widely used to solve a series of physical problems. Mathematically, the phase field models are generally derived from the functional variation of free energy. In general, the free energy $E(\phi)$ contains the sum of an integral phase of a nonlinear functional and a quadratic term:
\begin{equation}\label{intro-e1}
E(\phi)=\frac12(\phi,\mathcal{L}\phi)+E_1(\phi)=\frac12(\phi,\mathcal{L}\phi)+\int_\Omega F(\phi)d\textbf{x},
\end{equation}
where $\mathcal{L}$ is a symmetric non-negative linear operator, and $E_1(\phi)=\int_\Omega F(\phi)d\textbf{x}$ is nonlinear free energy. $F(\textbf{x})$ is the energy density function. The gradient flow from the energetic variation of the above energy functional $E(\phi)$ in \eqref{intro-e1} can be obtained as follows:
\begin{equation}\label{intro-e2}
\displaystyle\frac{\partial \phi}{\partial t}=-\mathcal{G}\mu,\quad\mu=\displaystyle\mathcal{L}\phi+F'(\phi),
\end{equation}
where $\mu=\frac{\delta E}{\delta \phi}$ is the chemical potential. $\mathcal{G}$ is a positive operator. For example, $\mathcal{G}=I$ for the $L^2$ gradient flow and $\mathcal{G}=-\Delta$ for the $H^{-1}$ gradient flow.

It is not difficult to find that the above phase field system satisfies the following energy dissipation law:
\begin{equation*}
\frac{d}{dt}E=(\frac{\delta E}{\delta \phi},\frac{\partial\phi}{\partial t})=-(\mathcal{G}\mu,\mu)\leq0,
\end{equation*}
which is a very important property for gradient flows in physics and mathematics. From a mathematical point of view, whether the numerical methods can maintain the discrete energy dissipation law is an important stability indicator. Up to now, many scholars considered a series of efficient and popular time discretized approaches to construct energy stable schemes for different phase field models such as convex splitting approach \cite{eyre1998unconditionally,shen2012second,shin2016first}, linear stabilized approach \cite{shen2010numerical,yang2017numerical}, exponential time differencing (ETD) approach \cite{du2019maximum,WangEfficient}, invariant energy quadratization (IEQ) approach \cite{chen2019efficient,chen2019fast,yang2016linear}, scalar auxiliary variable (SAV) approach \cite{xiaoli2019energy,shen2018scalar,ShenA} and so on. There are many different advantages for these methods. In recent years, SAV approach has become a very efficient and powerful way to construct energy stable schemes for gradient flows. Until now, it has been applied successfully to simulate many classical gradient flows such as Allen-Cahn models \cite{yang2020convergence,shen2018scalar}, Cahn-Hilliard models \cite{yang2018numerical,li2019energy}, phase field crystal models \cite{liu2019efficient,li2020stability}, molecular beam epitaxial growth model \cite{YangNumerical,cheng2019highly}, Cahn-Hilliard-Navier-Stokes models \cite{li2019sav} and so on. It is worth mentioning that some fluid mechanics problems such as Navier-Stokes models \cite{li2020new,lin2019numerical,xiaoli2020error} can also be simulated effectively by the SAV approach.

In order to motivate our improvements, we briefly review below the SAV approach. The key of the SAV approach is to transform the nonlinear potential $E_1(\phi)$ into a simple scalar quadratic form. In particular, assuming that $E_1(\phi)$ is bounded from below which means that there exists a constant $C$ to make $E_1(\phi)+C>0$. Define a scalar auxiliary variable
\begin{equation*}
r(t)=\sqrt{E_1(\phi)+C}=\sqrt{\int_\Omega F(\phi)d\textbf{x}+C}>0.
\end{equation*}
We then rewrite the gradient flow \eqref{intro-e2} as the following equivalent system with SAV:
\begin{equation}\label{intro-e3}
  \left\{
   \begin{array}{rll}
\displaystyle\frac{\partial \phi}{\partial t}&=&-\mathcal{G}\mu,\\
\mu&=&\displaystyle\mathcal{L}\phi+\frac{r}{\sqrt{E_1(\phi)+C}}F'(\phi),\\
r_t&=&\displaystyle\frac{1}{2\sqrt{E_1(\phi)+C}}\int_{\Omega}F'(\phi)\phi_td\textbf{x}.
   \end{array}
   \right.
\end{equation}

The above equivalent system \eqref{intro-e3} is very easy to construct linear and unconditional energy stable scheme. For example, a first-order semi-discrete scheme reads as follows
\begin{equation}\label{intro-e4}
  \left\{
   \begin{array}{rll}
\displaystyle\frac{\phi^{n+1}-\phi^{n}}{\Delta t}&=&-\mathcal{G}\mu^{n+1},\\
\mu^{n+1}&=&\displaystyle\mathcal{L}\phi^{n+1}+\frac{r^{n+1}}{\sqrt{E_1(\phi^{n})+C}}F'(\phi^{n}),\\
\displaystyle\frac{r^{n+1}-r^n}{\Delta t}&=&\displaystyle\frac{1}{2\sqrt{E_1(\phi^{n})+C}}\int_{\Omega}F'(\phi^{n})\frac{\phi^{n+1}-\phi^{n}}{\Delta t}d\textbf{x}.
   \end{array}
   \right.
\end{equation}

It is not difficult to obtain the following discrete dissipation law:
\begin{equation*}
\aligned
\left[\frac12(\mathcal{L}\phi^{n+1},\phi^{n+1})+|r^{n+1}|^2\right]-\left[\frac12(\mathcal{L}\phi^{n},\phi^{n})+|r^{n}|^2\right]\leq-\Delta t(\mathcal{G}\mu^{n+1},\mu^{n+1})\leq0.
\endaligned
\end{equation*}
No restriction on the time step is required.

Although the SAV method is very concise and effective, there still are several shortcomings need modifying and improving. In general, the scholars consider two aspects to modify the traditional SAV approach. One is to change the definition of the introduced SAV. In \cite{yang2020roadmap}, the authors introduced the generalized auxiliary variable method for devising energy stable schemes for general dissipative systems. A new Lagrange Multiplier approach for gradient flows is considered in \cite{cheng2020new}. We also propose an exponential SAV approach in \cite{liu2020exponential} to modify the traditional method to construct energy stable schemes. Another way of the modified method is to construct new algorithms to save the computational cost. For the traditional SAV approach, we general need to compute an inner product $(b^{n},\phi^{n+1})$ before obtaining $\phi^{n+1}$ in \cite{shen2018scalar} where $b^n=\frac{F'(\phi^{n})}{\sqrt{E_1(\phi^{n})+C}}$. A more efficiency way to obtain solutions is to compute the following two linear equations with constant coefficients by setting $\phi^{n+1}=\phi_1^{n+1}+r^{n+1}\phi_2^{n+1}$:
\begin{equation*}
\aligned
(I+\Delta t\mathcal{G}\mathcal{L})\phi_1^{n+1}=\phi^n,\quad(I+\Delta t\mathcal{G}\mathcal{L})\phi_2^{n+1}=-\Delta t\mathcal{G}b^n.
\endaligned
\end{equation*}

In \cite{huang2020highly}, the authors consider a new technique by replacing the controlling factor $\xi=\frac{r}{\sqrt{E_1(\phi)+C}}$ by $\theta+(1-\theta)\xi$ where $\theta=1+O(\Delta t^k),$ $k\geq1$ to construct new SAV energy stable schemes. The new SAV approach only requires solving one linear system with constant coefficients at each time step.

In order to show and give a comparative study for our fast SAV approach, we provide below a brief review of the new SAV approach in \cite{huang2020highly} to construct energy stable schemes for gradient flows. Assume that the energy $E(\phi)$ is bounded from below which means that there is a constant $C>0$ to satisfy $E(\phi)+C>0$. We introduce a scalar auxiliary variable $R(t)=E(\phi)+C$ which
satisfies the following dynamical equation
\begin{equation*}
\aligned
\frac{dR(t)}{dt}=\frac{dE}{dt}=-(\mathcal{G}\mu,\mu)\leq0.
\endaligned
\end{equation*}
Define $\xi(t)=\frac{R(t)}{E(\phi)+C}$ and note that $\xi(t)\equiv1$ at a continuous level. We can then rewrite the gradient flow \eqref{intro-e2} as the following equivalent system with SAV
\begin{equation}\label{intro-e5}
  \left\{
   \begin{array}{rll}
\displaystyle\frac{\partial \phi}{\partial t}&=&-\mathcal{G}\mu,\\
\mu&=&\displaystyle\mathcal{L}\phi+\left[\theta+(1-\theta)\xi\right]F'(\phi),\\
R_t&=&-\xi(\mathcal{G}\mu,\mu),
   \end{array}
   \right.
\end{equation}
where $\theta(t)$ can be an arbitrary function at the continuous level. A first-order semi-implicit scheme can be given as follows:
\begin{equation}\label{intro-e6}
  \left\{
   \begin{array}{l}
\displaystyle\frac{\phi^{n+1}-\left[\theta^n+(1-\theta^n)\xi^{n+1}\right]\phi^{n}}{\Delta t}=-\mathcal{G}\mu^{n+1},\\
\mu^{n+1}=\displaystyle\mathcal{L}\phi^{n+1}+\left[\theta^n+(1-\theta^n)\xi^{n+1}\right]F'(\phi^{n}),\\
\displaystyle\frac{R^{n+1}-R^n}{\Delta t}=-\xi^{n+1}(\mathcal{G}\overline{\mu}^{n+1},\overline{\mu}^{n+1}),\\
\xi^{n+1}=\displaystyle\frac{R^{n+1}}{E(\overline{\phi}^{n+1})+C}.
   \end{array}
   \right.
\end{equation}

By giving an arbitrary function $\theta(t)$ to satisfy $\theta=1+O(\Delta t)$, we can direct to observe that
\begin{equation}\label{intro-e7}
\displaystyle\frac{\phi^{n+1}-\left[\theta^n+(1-\theta^n)\xi^{n+1}\right]\phi^{n}}{\Delta t}=
\displaystyle\frac{\phi^{n+1}-\phi^{n}}{\Delta t}+\frac{O(\Delta t)(1-\xi^{n+1})}{\Delta t}=\left.\frac{\partial \phi}{\partial t}\right|^{n+1}+O(\Delta t).
\end{equation}
Combining the first two equations in \eqref{intro-e6}, we can obtain the following linear matrix equation
\begin{equation*}
\aligned
(I+\Delta t\mathcal{G}\mathcal{L})\phi_1^{n+1}=\left[\theta^n+(1-\theta^n)\xi^{n+1}\right](\phi^n-\Delta t\mathcal{G}F'(\phi^n)),
\endaligned
\end{equation*}
which means that we only require solving one linear equation with constant coefficients (see more details in \cite{huang2020highly}).

By introducing a new SAV $\theta$, the new SAV approach enjoys the following remarkable properties: (1) it only requires solving one linear system with constant coefficients at each time step; (2) it only requires the energy functional $E(\phi)$ be bounded from below; (3) it is extendable to higher-order BDF type energy stable schemes. However, it also may bring some new problems. For any $\theta=1+O(\Delta t^m)$, we have $\theta+(1-\theta)\xi=1+O(\Delta t^m)+O(\Delta t^m)\xi$. For the new $k$th order BDF$k$ scheme, we will discretize the nonlinear term in \eqref{intro-e5} as follows
\begin{equation}\label{intro-e8}
\left[\theta^n+(1-\theta^n)\xi^{n+1}\right]F'(\phi^{*,n+1})=F'(\phi^{*,n+1})+\left[O(\Delta t^k)+O(\Delta t^k)\xi^{n+1}\right]F'(\phi^{*,n+1}).
\end{equation}
As we can see, the new introduced variable $\theta$ has to be discrete explicitly when considering fast computation. The only implicit part $O(\Delta t^k) \xi^{n+1}F^{\prime}(\phi^{*,n+1})$ in the above scheme is much weaker than the traditional SAV approach. If $\theta$ is not be given efficiently, the convergence accuracy of the new method may not be as good as that of the traditional SAV method. We need to focus on the following two issues: (1) whether the variable $\theta$ is essential to construct fast calculation scheme. (2) whether the requirement of the nonlinear energy functional $E_1(\phi)$ be bounded from
below could be removed.

In this paper, we propose a novel technique to construct a highly efficient and accurate exponential semi-implicit scalar auxiliary variable (ESI-SAV) appraoch. The main contributions in our work lie in the fact that the new proposed ESI-SAV method only needs to solve one linear equation with constant coefficients at each time step without imposing the bounded below restriction of nonlinear free energy potential which is more reasonable and effective for various dissipative systems. Meanwhile it is easy to construct first-order, second-order and higher-order unconditionally energy stable time-stepping schemes. Other than that, the new proposed ESI-SAV approach can be proved to be very easy to solve the non-gradient but dissipative system such as Navier-Stokes equation. To our knowledge, there is no careful research on one step but energy stable and accurate SAV scheme without imposing the bounded below restriction.

Compared with our original E-SAV approach in \cite{liu2020exponential}, the new proposed ESI-SAV method has two main advantages: (1) the implicit discrete of the introduced scalar auxiliary variable $R$ guarantees that $R$ is bounded from below which adding the stability and robustness of the discrete scheme; (2) the new ESI-SAV approach is very easy to construct higher-order unconditionally energy stable time-stepping schemes.

The paper is organized as follows. In Sect.2, we consider a new procedure to obtain an energy stable ESI-SAV approach and construct first-order, second-order and higher-order time discrete schemes. In Sect.3, we use the same technique of the proposed ESI-SAV approach to solve the general nonlinear dissipative system suche as Navier-Stokes equations. Finally, in Sect.4, various 2D numerical simulations are demonstrated to verify the accuracy and efficiency of our proposed schemes.

\section{ESI-SAV approach for phase field models}
In this section, we will consider an exponential semi-implicit scalar auxiliary variable (ESI-SAV) approach for phase field models to construct energy stable numerical schemes. Exponential function is a special function that keeps the range constant positive. Thus, we introduce an exponential scalar auxiliary variable (E-SAV):
\begin{equation}\label{esav-e1}
\aligned
R(t)=\exp\left(E(\phi)\right)=\exp\left((\phi,\mathcal{L}\phi)+\int_\Omega F(\phi)d\textbf{x}\right).
\endaligned
\end{equation}
It is obviously $R(t)>0$ for any $t$. Then, the nonlinear functional $F'(\phi)$ can be transformed into the following equivalent formulation:
\begin{equation*}
F'(\phi)=\frac{R}{R}F'(\phi)=\frac{R}{\exp\left(E(\phi)\right)}F'(\phi).
\end{equation*}

Thus, \eqref{intro-e3} can be rewritten as the following equivalent system:
\begin{equation}\label{esav-e2}
  \left\{
   \begin{array}{rll}
\displaystyle\frac{\partial \phi}{\partial t}&=&\mathcal{G}\mu,\\
\mu&=&\displaystyle\mathcal{L}\phi+\frac{R}{\exp\left(E(\phi)\right)}F'(\phi),\\
\displaystyle\frac{dR}{dt}&=&-R(\mathcal{G}\mu,\mu).
   \end{array}
   \right.
\end{equation}

Noting that $R=\exp(E(\phi))>0$, we are easy to obtain the following inequality from the third equation in \eqref{esav-e2}:
\begin{equation*}
\frac{dR}{dt}=-R(\mathcal{G}\mu,\mu)=-\exp(E(\phi))(\mathcal{G}\mu,\mu)\leq0.
\end{equation*}

Noticing that $\ln(R)=\ln(\exp(E(\phi)))=E(\phi)$, we can obtain the original energy dissipation law:
\begin{equation*}
\frac{dE}{dt}=\frac{d\ln(R)}{dt}=\frac{1}{R}\frac{dR}{dt}=-(\mathcal{G}\mu,\mu)\leq0.
\end{equation*}

Next, we will consider some numerical schemes to illustrate that the proposed ESI-SAV approach is very easy to obtain linear and unconditionally energy stable schemes. More importantly, it can be found that both first-order, second-order and some higher-order semi-implicit numerical schemes with unconditionally energy stability can be constructed easily.

Before giving a semi-discrete formulation, we let $N>0$ be a positive integer and set
\begin{equation*}
\Delta t=T/N,\quad t^n=n\Delta t,\quad \text{for}\quad n\leq N.
\end{equation*}
\subsection{The first-order scheme}
A first-order scheme for solving the system \eqref{esav-e2} can be readily derived by the backward Euler method. The first-order scheme can be written as follows:
\begin{equation}\label{esav-first-e1}
  \left\{
   \begin{array}{rll}
\displaystyle\frac{\phi^{n+1}-\phi^{n}}{\Delta t}&=&\mathcal{G}\mu^{n+1},\\
\mu^{n+1}&=&\displaystyle\mathcal{L}\phi^{n+1}+\frac{R^{n+1}}{\exp\left(E(\phi^n)\right)}F'(\phi^n),\\
\displaystyle\frac{R^{n+1}-R^n}{\Delta t}&=&-R^{n+1}(\mathcal{G}\overline{\mu}^n,\overline{\mu}^n),
   \end{array}
   \right.
\end{equation}
with the initial conditions
\begin{equation*}
\phi^0=\phi_0(x,t),\quad R^0=\exp(E(\phi^0)).
\end{equation*}

Firstly, we obtain $\overline{\mu}^n$ as follows
\begin{equation}\label{esav-first-e2}
\overline{\mu}^{n}=\mathcal{L}\phi^{n}+F'(\phi^n).
\end{equation}

Then the scalar auxiliary variable $R^{n+1}$ can be solved by the third equation in \eqref{esav-first-e1}:
\begin{equation}\label{esav-first-e3}
R^{n+1}=\frac{R^n}{1+\Delta t(\mathcal{G}\mu^n,\mu^n)}.
\end{equation}

Then $\phi^{n+1}$ can be solved by the following linear equation:
\begin{equation}\label{esav-first-e4}
(I-\Delta t\mathcal{G}\mathcal{L})\phi^{n+1}=\phi^n+\Delta t\frac{R^{n+1}}{\exp\left(E(\phi^n)\right)}\mathcal{G}F'(\phi^n).
\end{equation}

To summarize, the first-order scheme \eqref{esav-first-e1} can be implemented as follows:
\begin{enumerate}
  \item[] $\bullet$ set $\overline{\mu}^n=\mathcal{L}\phi^{n}+F'(\phi^n)$ and compute $R^{n+1}$ from \eqref{esav-first-e3};
  \item[] $\bullet$ update $\phi^{n+1}$ from \eqref{esav-first-e4} and go to the next time step.
\end{enumerate}
The first-order ESI-SAV scheme \eqref{esav-first-e1} is much easier to implement than the traditional SAV. We observe that $\phi^{n+1}$ and $R^{n+1}$ can be solved step by step which means the above procedure only requires solving one linear equation with constant coefficients as in a standard semi-implicit scheme. As for the energy stability, we have the following theorem.
\begin{theorem}\label{esav-th1}
Given $R^n>0$, we then obtain $R^{n+1}>0$. The scheme \eqref{esav-first-e1} for the equivalent phase field system \eqref{esav-e2} is unconditionally energy stable in the sense that
\begin{equation*}
\aligned
R^{n+1}-R^n=-\Delta tR^{n+1}(\mathcal{G}\overline{\mu}^n,\overline{\mu}^n)\leq0.
\endaligned
\end{equation*}
and more importantly we have
\begin{equation*}
\aligned
\ln R^{n+1}-\ln R^n\leq0.
\endaligned
\end{equation*}
\end{theorem}
\begin{proof}
Using the definition of $R(t)$, we can obtain $R^0=\exp(\phi^0)>0$. Noting that $R^{n+1}=\frac{R^n}{1+\Delta t(\mathcal{G}\overline{\mu}^n,\overline{\mu}^n)}$, we can immediately obtain $R^{n+1}>0$ for all $n\geq1$ because of $R^0>0$ and $(\mathcal{G}\overline{\mu}^n,\overline{\mu}^n)\geq0$.

Combining the inequality $R^{n+1}>0$ with the third equation in scheme \eqref{esav-first-e1}, we can easy to obtain the following modified energy stability:
\begin{equation*}
\aligned
R^{n+1}-R^n=-\Delta tR^{n+1}(\mathcal{G}\overline{\mu}^n,\overline{\mu}^n)\leq0.
\endaligned
\end{equation*}
Noticing that $E(\phi)=\ln(\exp(E(\phi)))=\ln(R)$ and the logarithm function is strictly monotone increasing function, we can also obtain the following energy stability:
\begin{equation*}
\aligned
\ln R^{n+1}-\ln R^n\leq0.
\endaligned
\end{equation*}
\end{proof}
\begin{remark}\label{essav-re1}
To obtain $R^1$, we need to simulate $\frac{dE}{dt}$ as $-(\mathcal{G}\overline{\mu}^0,\overline{\mu}^0)$ from the scheme \eqref{esav-first-e1}. It is seemly not reasonable to obtain a good approximation for $\frac{dE}{dt}$ at $t^1$ if we just use $\phi^0$ or the initial energy $E^0$. An appropriate way to simulate $\frac{dE}{dt}$ as $-(\mathcal{G}\overline{\mu}^1,\overline{\mu}^1)$ where $\overline{\mu}^1=\mathcal{L}\overline{\phi}^{1}+F'(\overline{\phi}^1)$ and $(I-\Delta t\mathcal{G}\mathcal{L})\overline{\phi}^{1}=\phi^0+\Delta t\mathcal{G}F'(\phi^0)$.
\end{remark}
\begin{remark}\label{esav-re2}
To prevent the solution "blowing up" because of the exponential function increasing rapidly, we can add a positive constant $C$ to redefine the exponential scalar auxiliary variable:
\begin{equation*}
\aligned
R(t)=\exp\left(\frac{E(\phi)}{C}\right)=\exp\left(\frac{1}{C}(\phi,\mathcal{L}\phi)+\frac{1}{C}\int_\Omega F(\phi)d\textbf{x}\right).
\endaligned
\end{equation*}
The energy dissipation law is also keep original at the continuous level although a positive constant $C$ is added to the variable $R$:
\begin{equation*}
\frac{dE}{dt}=\frac{Cd\ln(R)}{dt}=\frac{C}{R}\frac{dR}{dt}=-(\mathcal{G}\mu,\mu)\leq0.
\end{equation*}

\end{remark}
\subsection{The second-order Crank-Nicolson scheme}
In this subsection, we will consider a linear, second-order, sequentially solved and unconditionally stable exponential semi-implicit SAV (ESI-SAV) scheme. Firstly, we introduce a new variable $\xi=\frac{R}{\exp(E(\phi))}$. It is obviously $\xi\equiv1$ at the continuous level. Meanwhile, $\xi(2-\xi)$ is also equal to 1 at the continuous level. Then the phase field system \eqref{intro-e3} can be rewritten as the following equivalent system:
\begin{equation}\label{esav-second-system}
  \left\{
   \begin{array}{rll}
\displaystyle\frac{\partial \phi}{\partial t}&=&\mathcal{G}\mu,\\
\mu&=&\displaystyle\mathcal{L}\phi+\xi(2-\xi)F'(\phi),\\
\xi&=&\displaystyle\frac{R}{\exp(E(\phi))},\\
\displaystyle\frac{dR}{dt}&=&-R(\mathcal{G}\mu,\mu).
   \end{array}
   \right.
\end{equation}
The above new equivalent system is also keep the original energy dissipation law:
\begin{equation*}
\frac{dE}{dt}=\frac{d\ln(R)}{dt}=\frac{1}{R}\frac{dR}{dt}=-(\mathcal{G}\mu,\mu)\leq0.
\end{equation*}

A linear, exponential semi-implicit SAV scheme based on the second order Crank-Nicolson formula (CN) for \eqref{esav-second-system} reads as: for $n\geq1$,
\begin{equation}\label{esav-second-e1}
  \left\{
   \begin{array}{rll}
\displaystyle\frac{\phi^{n+1}-\phi^{n}}{\Delta t}&=&\mathcal{G}\mu^{n+\frac12},\\
\mu^{n+\frac12}&=&\displaystyle\mathcal{L}\frac{\phi^{n+1}+\phi^{n}}{2}+\xi^{n+1}(2-\xi^{n+1})F'(\phi^{*,n+\frac{1}{2}}),\\
\xi^{n+1}&=&\displaystyle\frac{R^{n+1}}{\exp(\phi^{*,n+\frac{1}{2}})},\\
\displaystyle\frac{R^{n+1}-R^n}{\Delta t}&=&-R^{n+1}(\mathcal{G}\overline{\mu}^{n+\frac{1}{2}},\overline{\mu}^{n+\frac{1}{2}}),
   \end{array}
   \right.
\end{equation}
where $\phi^{*,n+\frac{1}{2}}$ is any explicit $O(\Delta t^2)$ approximation for $\phi(t^{n+\frac{1}{2}})$ which can be flexible according to the problem. Here, we choose
\begin{equation}\label{esav-second-e2}
\aligned
&\phi^{*,n+\frac{1}{2}}=\frac32\phi^n-\frac12\phi^{n-1}, \quad n\geq1,\\
\endaligned
\end{equation}
and for $n=0$, we compute $\phi^{*,n+\frac{1}{2}}$ as follows:
\begin{equation}\label{esav-second-e3}
\aligned
&\displaystyle\frac{\phi^{*,\frac{1}{2}}-\phi^0}{(\Delta t)/2}=\mathcal{G}\left[\mathcal{L}\phi^{*,\frac{1}{2}}+F^{'}(\phi^0)\right],
\endaligned
\end{equation}
which has a local truncation error of $O(\Delta t^2)$.

Noting that the fourth discrete scheme for $R$ in \eqref{esav-second-e1} is a first-order scheme which means $R^{n+1}=R(t^{n+1})+O(\Delta t)$, we then obtain
\begin{equation*}
\xi^{n+1}=\xi(t^{n+1})+C_1\Delta t=1+C_1\Delta t.
\end{equation*}
Then, we can obtain the following equation
\begin{equation*}
\xi^{n+1}(2-\xi^{n+1})=(1+C_1\Delta t)(1-C_1\Delta t)=1-C_1^2\Delta t^2.
\end{equation*}
which means the nonlinear term of $\xi$ in the second equation in \eqref{esav-second-e1} can be treated as the second-order approximation of 1.

The Crank-Nicolson scheme also enjoys the same stability as the first-order scheme \eqref{esav-first-e1}, namely, we can prove the following result using exactly the same procedure.
\begin{theorem}\label{esav-th2}
Given $R^n>0$, we then obtain $R^{n+1}>0$. The second-order Crank-Nicolson scheme \eqref{esav-second-e1} is unconditionally energy stable in the sense that
\begin{equation*}
\aligned
R^{n+1}-R^n=-\Delta tR^{n+1}(\mathcal{G}\overline{\mu}^{n+\frac{1}{2}},\overline{\mu}^{n+\frac{1}{2}})\leq0,
\endaligned
\end{equation*}
and
\begin{equation*}
\aligned
\ln R^{n+1}-\ln R^n\leq0.
\endaligned
\end{equation*}
\end{theorem}
\begin{remark}\label{esav-re1}
The second-order ES-SAV scheme \eqref{esav-second-e1} based on Crank-Nicolson can be implemented sequentially as follows:
\begin{enumerate}
  \item[(i)] Compute the initial values of $\phi^0$ and $R^0=\exp(E(\phi^0))$;
  \item[(ii)] Compute $\phi^{*,\frac{1}{2}}$ from \eqref{esav-second-e3} and set n=0;
  \item[(iii)] set $\overline{\mu}^{n+\frac{1}{2}}=\mathcal{L}\phi^{*,n+\frac{1}{2}}+F'(\phi^{*,n+\frac{1}{2}})$ and compute $R^{n+1}$ from \eqref{esav-first-e3};
  \item[(iv)] set n=n+1;
  \item[(v)] Compute $\xi^n$ from the third equation in \eqref{esav-second-e1};
  \item[(vi)] Compute $\phi^n$ from the first equation in \eqref{esav-second-e1};
  \item[(vii)] Compute $\phi^{*,n+\frac{1}{2}}$ from \eqref{esav-second-e2} and go back to step (iii).
\end{enumerate}
\end{remark}
\subsection{The high-order BDF$k$ scheme}
From above second-order Crank-Nicolson scheme \eqref{esav-second-e1}, we observe that if we give a proper discretization of $\xi$, the first-order scheme for $R$ doesn't have any effect on the order of $\phi$. we can achieve overall $k$th-order accuracy coupled with $k$-step BDF for $\phi$ by using just a first-order approximation for $R$. To give a unified framework for the high-order BDF$k$ scheme uniformly and concisely, we introduce a new function $V(\xi)$ which is equal to 1 at the continuous level. Then, we can obtain the following equivalent system of \eqref{intro-e3}:
\begin{equation}\label{esav-BDF-system}
  \left\{
   \begin{array}{rll}
\displaystyle\frac{\partial \phi}{\partial t}&=&\mathcal{G}\mu,\\
\mu&=&\displaystyle\mathcal{L}\phi+V(\xi)F'(\phi),\\
\xi&=&\displaystyle\frac{R}{\exp(E(\phi))},\\
\displaystyle\frac{dR}{dt}&=&-R(\mathcal{G}\mu,\mu).
   \end{array}
   \right.
\end{equation}
Set $V(\xi)=\xi(2-\xi)$, then it is not difficult to obtain the BDF2 scheme:
for $n\geq1$,
\begin{equation}\label{esav-BDF-e1}
   \begin{array}{rll}
\displaystyle\frac{3\phi^{n+1}-4\phi^{n}+\phi^{n-1}}{2\Delta t}&=&\mathcal{G}\mu^{n+1},\\
\mu^{n+1}&=&\displaystyle\mathcal{L}\phi^{n+1}+V(\xi^{n+1})F'(\phi^{*,n+1}),\\
V(\xi^{n+1})&=&\xi^{n+1}(2-\xi^{n+1}),\\
\xi^{n+1}&=&\displaystyle\frac{R^{n+1}}{\exp(\phi^{*,n+1})},\\
\displaystyle\frac{R^{n+1}-R^n}{\Delta t}&=&-R^{n+1}(\mathcal{G}\overline{\mu}^{n+1},\overline{\mu}^{n+1}),
   \end{array}
\end{equation}
where $\phi^{*,n+1}$ is any explicit $O(\Delta t^2)$ approximation for $\phi(t^{n+1})$ which can be flexible according to the problem. Here, we choose
\begin{equation}\label{esav-BDF-e2}
\aligned
&\phi^{*,n+1}=2\phi^n-\phi^{n-1}, \quad n\geq1,\\
&\overline{\mu}^{n+1}=\mathcal{L}\phi^{*,n+1}+F^{'}(\phi^{*,n+1}), \quad n\geq1.
\endaligned
\end{equation}
Set $V(\xi)=\xi(3-3\xi+\xi^2)$, then we can obtain the following BDF3 scheme: for $n\geq2$,
\begin{equation}\label{esav-BDF-e3}
   \begin{array}{rll}
\displaystyle\frac{11\phi^{n+1}-18\phi^{n}+9\phi^{n-1}-2\phi^{n-2}}{6\Delta t}&=&\mathcal{G}\mu^{n+1},\\
\mu^{n+1}&=&\displaystyle\mathcal{L}\phi^{n+1}+V(\xi^{n+1})F'(\phi^{*,n+1}),\\
V(\xi^{n+1})&=&\xi^{n+1}[3-3\xi^{n+1}+(\xi^{n+1})^2],\\
\xi^{n+1}&=&\displaystyle\frac{R^{n+1}}{\exp(\phi^{*,n+1})},\\
\displaystyle\frac{R^{n+1}-R^n}{\Delta t}&=&-R^{n+1}(\mathcal{G}\overline{\mu}^{n+1},\overline{\mu}^{n+1}),
   \end{array}
\end{equation}
where $\phi^{*,n+1}$ is any explicit $O(\Delta t^2)$ approximation for $\phi(t^{n+1})$ which can be chosen as follows
\begin{equation}\label{esav-BDF-e4}
\aligned
&\phi^{*,n+1}=3\phi^n-3\phi^{n-1}+\phi^{n-2}, \quad n\geq2.
\endaligned
\end{equation}
Set $V(\xi)=\xi(2-\xi)(2-2\xi+\xi^2)$, then we can also obtain the following BDF4 scheme: for $n\geq3$,
\begin{equation}\label{esav-BDF-e5}
   \begin{array}{rll}
\displaystyle\frac{25\phi^{n+1}-48\phi^{n}+36\phi^{n-1}-16\phi^{n-2}+3\phi^{n-3}}{12\Delta t}&=&\mathcal{G}\mu^{n+1},\\
\mu^{n+1}&=&\displaystyle\mathcal{L}\phi^{n+1}+V(\xi^{n+1})F'(\phi^{*,n+1}),\\
V(\xi^{n+1})&=&\xi^{n+1}(2-\xi^{n+1})[2-2\xi^{n+1}+(\xi^{n+1})^2],\\
\xi^{n+1}&=&\displaystyle\frac{R^{n+1}}{\exp(\phi^{*,n+1})},\\
\displaystyle\frac{R^{n+1}-R^n}{\Delta t}&=&-R^{n+1}(\mathcal{G}\overline{\mu}^{n+1},\overline{\mu}^{n+1}),
   \end{array}
\end{equation}
where $\phi^{*,n+1}$ is any explicit $O(\Delta t^2)$ approximation for $\phi(t^{n+1})$ which can be chosen as follows
\begin{equation}\label{esav-BDF-e6}
\aligned
&\phi^{*,n+1}=4\phi^n-6\phi^{n-1}+4\phi^{n-2}-\phi^{n-3}, \quad n\geq3.
\endaligned
\end{equation}
\begin{lemma}
The discretization of $V(\xi^{n+1})$ has the $k$-th accuracy of $1$ for the BDF$k$ scheme.
\end{lemma}
\begin{proof}
For the BDF2 scheme \eqref{esav-BDF-e1}, we have proved $V(\xi^{n+1})=\xi^{n+1}(2-\xi^{n+1})=1-C_1^2\Delta t^2.$ For the BDF3 scheme \eqref{esav-BDF-e3}, we have
\begin{equation*}
\aligned
V(\xi^{n+1})
&=\xi^{n+1}[3-3\xi^{n+1}+(\xi^{n+1})^2]\\
&=[1-(1-\xi^{n+1})][1+(1-\xi^{n+1})+(1-\xi^{n+1})^2]\\
&=(1-C_1\Delta t)[1+C_1\Delta t+(C_1\Delta t)^2]\\
&=1-C_1^3(\Delta t)^3.
\endaligned
\end{equation*}
And for the BDF4 scheme \eqref{esav-BDF-e5}, we have
\begin{equation*}
\aligned
V(\xi^{n+1})
&=\xi^{n+1}(2-\xi^{n+1})[2-2\xi^{n+1}+(\xi^{n+1})^2]\\
&=[1-(1-\xi^{n+1})][1+(1-\xi^{n+1})][1+(1-\xi^{n+1})^2]\\
&=(1-C_1\Delta t)(1+C_1\Delta t)[1+(C_1\Delta t)^2]\\
&=[1-C_1^2(\Delta t)^2][1+C_1^2(\Delta t)^2]\\
&=1-C_1^4(\Delta t)^4,
\endaligned
\end{equation*}
which completes the proof.
\end{proof}

The BDF$k$ schemes \eqref{esav-BDF-e1}-\eqref{esav-BDF-e5} also enjoy the same stability as the first-order scheme \eqref{esav-first-e1}, namely, we can prove the following result using exactly the same procedure.
\begin{theorem}\label{esav-th3}
Given $R^n>0$, we then obtain $R^{n+1}>0$. The BDF$k$ ESI-SAV schemes \eqref{esav-BDF-e1}-\eqref{esav-BDF-e5} are all unconditionally energy stable in the sense that
\begin{equation*}
\aligned
R^{n+1}-R^n=-\Delta tR^{n+1}(\mathcal{G}\overline{\mu}^{n+1},\overline{\mu}^{n+1})\leq0.
\endaligned
\end{equation*}
and
\begin{equation*}
\aligned
\ln R^{n+1}-\ln R^n\leq0.
\endaligned
\end{equation*}
\end{theorem}
\section{ESI-SAV approach for non-gradient but dissipative system}
In this section, we will show that the same technique of the proposed ESI-SAV approach can be used to solve the general nonlinear dissipative system. Consider a domain $\Omega$ in two or three dimensions and a dissipative system on this domain, whose dynamics is described by
\begin{equation}\label{nds-e1}
\aligned
\displaystyle\frac{\partial \textbf{u}}{\partial t}+A\textbf{u}+f(\textbf{u})=0,
\endaligned
\end{equation}
where $\textbf{u}(x,t)$ denotes the state variables of the system, $A$ is an elliptic operator. $f(\textbf{u})$ is an operator that gives rise to the dissipative dynamics of the system and generally would be nonlinear. The above system satisfies the following energy dissipative law
\begin{equation}\label{nds-e2}
\aligned
\displaystyle\frac{dE(\textbf{u})}{dt}=-(\mathcal{G}\textbf{u},\textbf{u}),
\endaligned
\end{equation}
where $\mathcal{G}$ is a positive operator.

The proposed ESI-SAV approach is very natural and efficient to solve this problem. Define an exponential scalar auxiliary variable:
\begin{equation*}
\aligned
R(t)=\exp\left(E(\textbf{u})\right).
\endaligned
\end{equation*}

We can rewrite the dissipative system \eqref{nds-e1} with $R(t)$ as follows
\begin{equation}\label{nds-e3}
\aligned
&\displaystyle\frac{\partial \textbf{u}}{\partial t}+A\textbf{u}+V(\xi)f(\textbf{u})=0,\\
&\displaystyle\frac{dR(t)}{dt}=-R(t)(\mathcal{G}\textbf{u},\textbf{u}),\\
&\displaystyle\xi=\frac{R(t)}{\exp\left(E(\textbf{u})\right)},
\endaligned
\end{equation}
where $V(\xi)$ is a function of $\xi$ which is different with different schemes.

A second-order scheme for solving the system \eqref{nds-e3} can be readily derived by the Crank-Nicolson method as follows:
\begin{equation}\label{nds-e4}
  \left\{
   \begin{array}{rll}
&\displaystyle\frac{\textbf{u}^{n+1}-\textbf{u}^n}{\Delta t}+A\textbf{u}^{n+1}+V(\xi^{n+1})f(\textbf{u}^{n+\frac{1}{2}})=0,\\
&\displaystyle\frac{R^{n+1}-R^n}{\Delta t}=-R^{n+1}(\mathcal{G}\overline{\textbf{u}}^{n+\frac{1}{2}},\overline{\textbf{u}}^{n+\frac{1}{2}}),\\
&\displaystyle\xi^{n+1}=\frac{R^{n+1}}{\exp\left(E(\textbf{u}^n)\right)},\\
&V(\xi^{n+1})=\xi^{n+1}(1-\xi^{n+1}),
   \end{array}
   \right.
\end{equation}
with the initial conditions
\begin{equation*}
\textbf{u}^0=\textbf{u}_0(x,t),\quad R^0=\exp(E(\textbf{u}^0)).
\end{equation*}

Similarly, Noting that $R^0>0$, we can obtain $R^n>0$ for all $n\geq0$. The second-order scheme \eqref{nds-e4} is unconditionally energy stable in the sense that
\begin{equation*}
\aligned
R^{n+1}-R^n=-\Delta tR^{n+1}(\mathcal{G}\overline{\textbf{u}}^{n+\frac{1}{2}},\overline{\textbf{u}}^{n+\frac{1}{2}})\leq0.
\endaligned
\end{equation*}
and
\begin{equation*}
\aligned
\ln R^{n+1}-\ln R^n\leq0.
\endaligned
\end{equation*}

We can also give energy stable second-order and higher-order schemes as before, but for the sake of brevity, we omit the detailed formula here.

Next, let's take the classic Navier-Stokes equation for example. Consider the following incompressible Navier-Stokes equations in $\Omega\times J$:
\begin{equation}\label{ns-e1}
\aligned
&\displaystyle\frac{\partial\textbf{u}}{\partial t}+\textbf{u}\cdot\nabla\textbf{u}-\nu\Delta\textbf{u}+\nabla p=0,\\
&\nabla\cdot\textbf{u}=0,\\
&\textbf{u}|_{\partial \Omega}=0,
\endaligned
\end{equation}
where the domain $\Omega$ is in two or three dimensions with a sufficiently smooth boundary and $J=(0, T]$, $\textbf{u}$ and $p$ are the normalized velocity and pressure, $\nu>0$ denotes the inverse of the Reynolds number.

Consider the total energy $E(\textbf{u})=\int_{\Omega}\frac12|\textbf{u}|^2$, then the above system satisfies the following property:
\begin{equation}\label{ns-e2}
\aligned
\displaystyle\frac{dE}{dt}=(\textbf{u}_t,\textbf{u})=-(\nu\nabla\textbf{u},\nabla\textbf{u}),
\endaligned
\end{equation}
which means the above Navier-Stokes equations are energy dissipative system.

Next, we will consider an energy stable numerical methods based on proposed ESI-SAV approach. Introduce a similar exponential scalar auxiliary variable:
\begin{equation}\label{ns-e3}
\aligned
R(t)=\exp\left(E(\textbf{u})\right)=\exp\left(\int_{\Omega}\frac12|\textbf{u}|^2\right).
\endaligned
\end{equation}
We then rewrite the incompressible Navier-Stokes equations \eqref{ns-e1} as the following equivalent system with SAV:
\begin{equation}\label{ns-e4}
\aligned
&\displaystyle\frac{\partial\textbf{u}}{\partial t}+\frac{R(t)}{\exp\left(E(\textbf{u})\right)}\textbf{u}\cdot\nabla\textbf{u}-\nu\Delta\textbf{u}+\nabla p=0,\\
&\nabla\cdot\textbf{u}=0,\\
&\textbf{u}|_{\partial \Omega}=0,\\
&\displaystyle\frac{dR(t)}{dt}=-R(t)(\nu\nabla\textbf{u},\nabla\textbf{u}).
\endaligned
\end{equation}

The first-order semi-discrete energy stable ESI-SAV version can be written as follows: given $\textbf{u}^0$ and $R^0=\exp(E(\textbf{u}^0))$, find $(u^{n+1},p^{n+1},R^{n+1})$ by solving
\begin{equation}\label{ns-e5}
\aligned
&\displaystyle\frac{\textbf{u}^{n+1}-\textbf{u}^{n}}{\Delta t}+\frac{R^{n+1}}{\exp\left(E(\textbf{u}^n)\right)}\textbf{u}^n\cdot\nabla\textbf{u}^n-\nu\Delta\textbf{u}^{n+1}+\nabla p^{n+1}=0,\\
&\nabla\cdot\textbf{u}^{n+1}=0,\\
&\textbf{u}^{n+1}|_{\partial \Omega}=0,\\
&\displaystyle\frac{R^{n+1}-R^n}{\Delta t}=-R^{n+1}(\nu\nabla\overline{\textbf{u}}^n,\nabla\overline{\textbf{u}}^n).
\endaligned
\end{equation}
\section{Examples and discussion}
In this section, several numerical examples are given to demonstrate the accuracy, energy stability and efficiency of the proposed exponential semi-implicit SAV schemes when applying to the some classical phase field models such as the Allen-Cahn model, Cahn-Hilliard model, phase field crystal model and so on. A comparative study in accuracy of the traditional SAV approach, the new SAV approach in \cite{huang2020highly}, usual semi-implicit approach and our proposed exponential semi-implicit SAV approach are considered to show the accuracy and efficiency. In all examples, we consider the periodic boundary conditions and use a Fourier spectral method in space.

\subsection{Allen-Cahn and Cahn-Hilliard equations}
Allen-Cahn and Cahn-Hilliard equations are very classical phase field models and have been widely used in many fields involving physics, materials science, finance and image processing \cite{chen2018accurate,chen2018power,du2018stabilized}.

In general, both Allen-Cahn for $\mathcal{G}=I$ and Cahn-Hilliard models for $\mathcal{G}=-\Delta$ can be given as follows
\begin{equation}\label{section5_e_model}
  \left\{
   \begin{array}{rlr}
\displaystyle\frac{\partial \phi}{\partial t}&=-\mathcal{G}\mu,     &(\textbf{x},t)\in\Omega\times J,\\
                                          \mu&=-\epsilon^2\Delta \phi+F'(\phi),&(\textbf{x},t)\in\Omega\times J,
   \end{array}
   \right.
\end{equation}
where $J=(0,T]$, $\mu=\frac{\delta E}{\delta \phi}$ is the chemical potential and the energy $E$ is the following Lyapunov energy functional:
\begin{equation}\label{section5_energy1}
E(\phi)=\int_{\Omega}(\frac{\epsilon^2}{2}|\nabla \phi|^2+F(\phi))d\textbf{x},
\end{equation}
where the most commonly used form Ginzburg-Landau double-well type potential is defined as $F(\phi)=\frac{1}{4}(\phi^2-1)^2$.

At first, we give the following example to test the accuracy and efficiency of the proposed ESI-SAV. The numerical schemes we tested and compared are the traditional SAV, the new SAV in \cite{huang2020highly}, usual semi-implicit and the proposed ESI-SAV approaches. We use the Fourier spectral Galerkin method for spatial discretization with $N=128$. The true solution is unknown and we therefore use the Fourier Galerkin approximation in the case $\Delta t=1e-4$ as a reference solution.

\textbf{Example 1}: Consider the above Allen-Cahn equation in $\Omega=[0,2\pi]^2$ with $\epsilon=0.1$, and the following initial condition \cite{ShenA}:
\begin{equation*}
\aligned
\phi(x,y,0)=0.1\cos(x)\cos(y).
\endaligned
\end{equation*}

We first set $T=10$ and use the first-order time discrete schemes based on SAV approach, new SAV approach in \cite{huang2020highly}, usual semi-implicit approach and our proposed exponential semi-implicit SAV approach. For the new SAV scheme, we need to give $\theta=1+\Delta t$ to satisfy $\theta=1+O(\Delta t)$. The computational error and convergence rates are shown in Table \ref{tab:tab1}. One can see that both accuracy and convergence rates for the SAV approach, the new SAV approach and semi-implicit (SEMI) approach are nearly same. The considered ESI-SAV method seemly has better accuracy. The second-order CN and BDF2 schemes, the third-order BDF3 scheme and the fourth-order BDF4 scheme based on the proposed ESI-SAV method are also given in Figure \ref{fig:fig2} for $T=2$, where we can observe the expected convergence rate of the field variable $\phi$ for all cases. In Figure \ref{fig:fig1}, we give the time evolution of $1-V(\xi)$ for the ESI-SAV CN, BDF2, BDF3 and BDF4 schemes with $\Delta t=0.1$. The errors between 1 and $V(\xi)$ show that the introduced functions $V(\xi)$ are close enough to 1. Figure \ref{fig:energy} shows the time histories of the modified energy $R$ and $\ln R$, and the original energy $E$ obtained by the current ESI-SAV method using time step $\Delta t=0.01$. It can be observed that all the energy history curves decrease dramatically at the beginning and level off gradually, indicating the stability of the proposed method.
\begin{table}[h!b!p!]
\small
\centering
\caption{\small The $L^\infty$ errors, convergence rates for first-order scheme in time for SAV approach, new SAV approach with $\theta=1+\Delta t$, usual semi-implicit (SEMI) approach and exponential semi-implicit SAV (ESI-SAV) approach of Allen-Cahn equation with $T=10$. }\label{tab:tab1}
\begin{tabular}{|c|c|c|c|c|c|c|c|c|}
\hline
$T=10$&\multicolumn{2}{c|}{SAV}&\multicolumn{2}{c|}{NSAV}&\multicolumn{2}{c|}{SEMI}&\multicolumn{2}{c|}{ESI-SAV}\\
\cline{1-9}
$\Delta t$&Error&Rate&Error&Rate&Error&Rate&Error&Rate\\
\cline{1-9}
$\frac12$      &1.0961e-2   &---   &1.1674e-2   &---    &1.1643e-2   &---   &2.5907e-3   &---   \\
$\frac14$      &2.6521e-3   &2.0471&3.3084e-3   &1.8191 &3.2859e-3   &1.8251&1.2967e-3   &0.9985\\
$\frac18$      &1.0350e-3   &1.3573&1.2287e-3   &1.4290 &1.2069e-3   &1.4449&6.4678e-4   &1.0034\\
$\frac{1}{16}$ &4.9722e-4   &1.0576&5.3544e-4   &1.1983 &5.1386e-4   &1.2318&3.2103e-4   &1.0105\\
$\frac{1}{32}$ &2.4608e-4   &1.0147&2.5558e-4   &1.0669 &2.3431e-4   &1.1329&1.5796e-4   &1.0231\\
$\frac{1}{64}$ &1.1988e-4   &1.0375&1.2962e-4   &0.9795 &1.0909e-4   &1.1028&7.6388e-5   &1.0481\\
\hline
\end{tabular}
\end{table}
\begin{table}[h!b!p!]
\small
\centering
\caption{\small The $L^\infty$ errors, convergence rates for second-order CN and BDF2 schemes, the third-order BDF3 scheme and the fourth-order BDF4 scheme based on the proposed ESI-SAV method of Allen-Cahn equation with $T=2$. }\label{tab:tab2}
\begin{tabular}{|c|c|c|c|c|c|c|c|c|}
\hline
$T=2$&\multicolumn{2}{c|}{CN}&\multicolumn{2}{c|}{BDF2}&\multicolumn{2}{c|}{BDF3}&\multicolumn{2}{c|}{BDF4}\\
\cline{1-9}
$\Delta t$&Error&Rate&Error&Rate&Error&Rate&Error&Rate\\
\cline{1-9}
$\frac14$      &4.5903e-2   &---   &1.2860e-1   &---    &2.9698e-2   &---   &6.2867e-3   &---   \\
$\frac18$      &1.3035e-2   &1.8162&3.9131e-2   &1.7165 &4.7478e-3   &2.6450&4.4096e-4   &3.8336\\
$\frac116$     &3.4522e-3   &1.9168&1.0742e-2   &1.8650 &6.6255e-4   &2.8411&2.8414e-5   &3.9559\\
$\frac{1}{32}$ &8.8602e-4   &1.9621&2.8060e-3   &1.9366 &8.7241e-5   &2.9249&1.7923e-6   &3.9867\\
$\frac{1}{64}$ &2.2371e-4   &1.9857&7.1463e-4   &1.9732 &1.1182e-5   &2.9638&1.1233e-7   &3.9959\\
$\frac{1}{128}$&5.5587e-5   &2.0088&1.7833e-4   &2.0026 &1.4124e-6   &2.9849&6.9969e-9   &4.0048\\
\hline
\end{tabular}
\end{table}
\begin{figure}[htp]
\centering
\includegraphics[width=8cm,height=6cm]{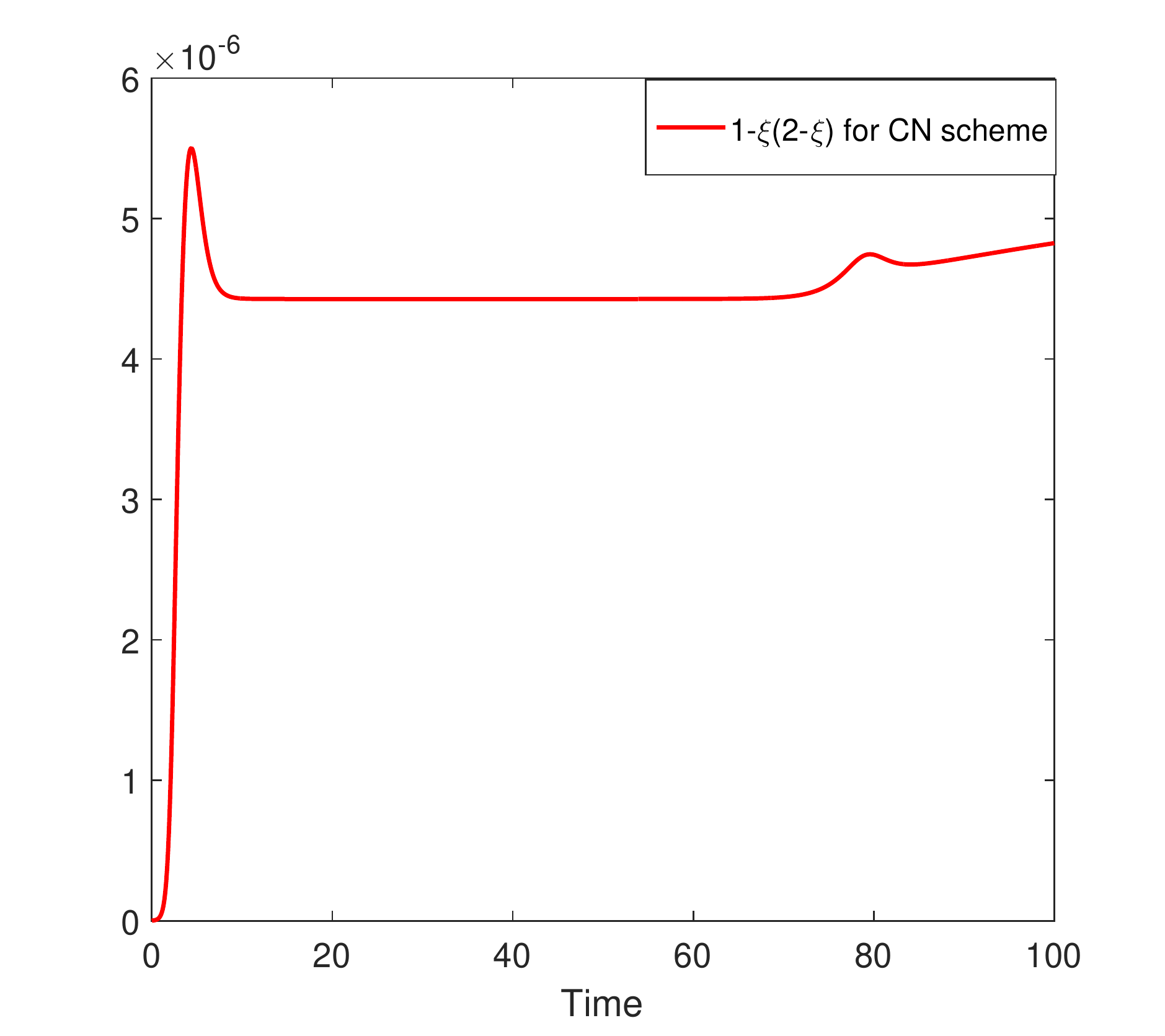}
\includegraphics[width=8cm,height=6cm]{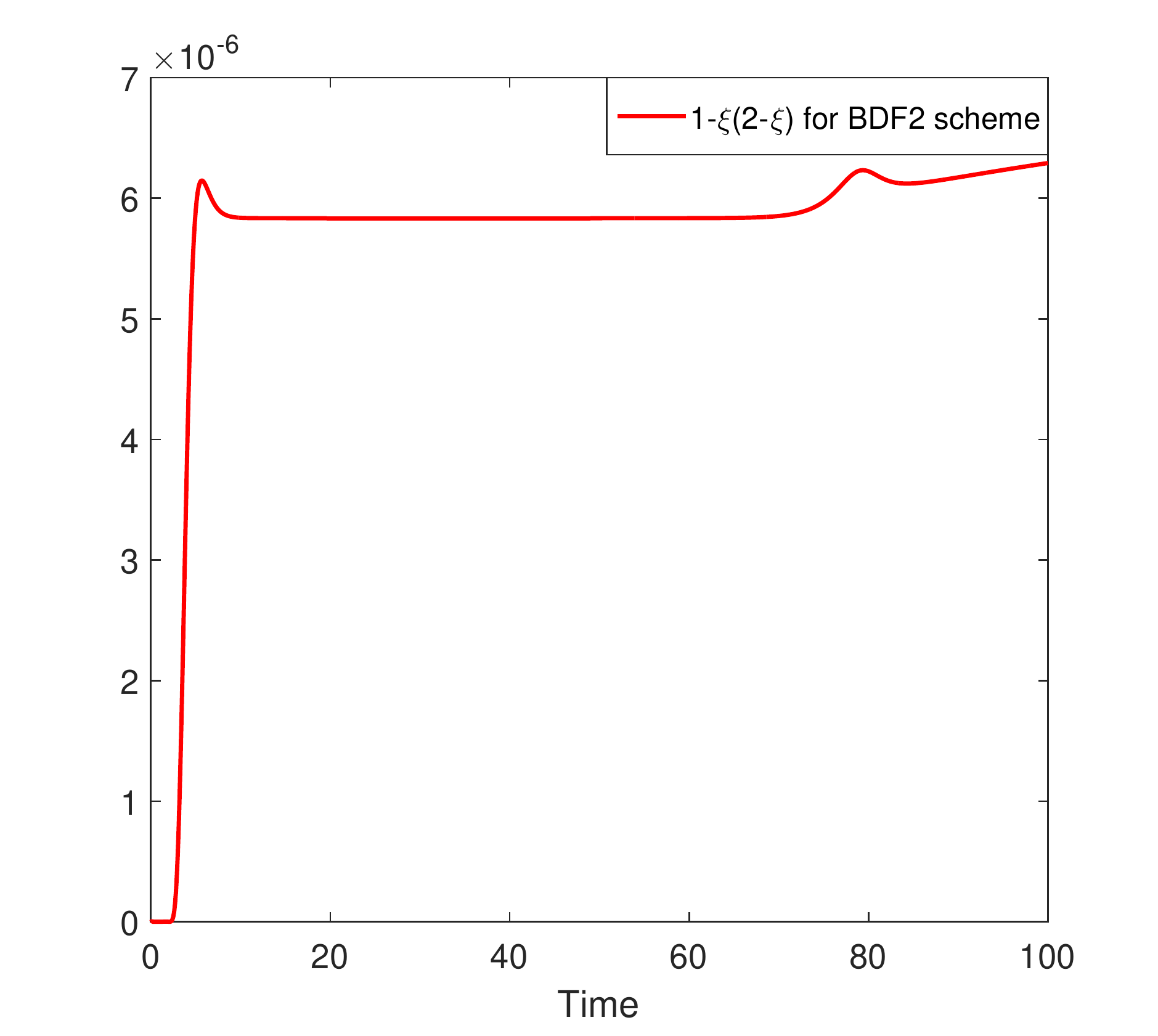}
\includegraphics[width=8cm,height=6cm]{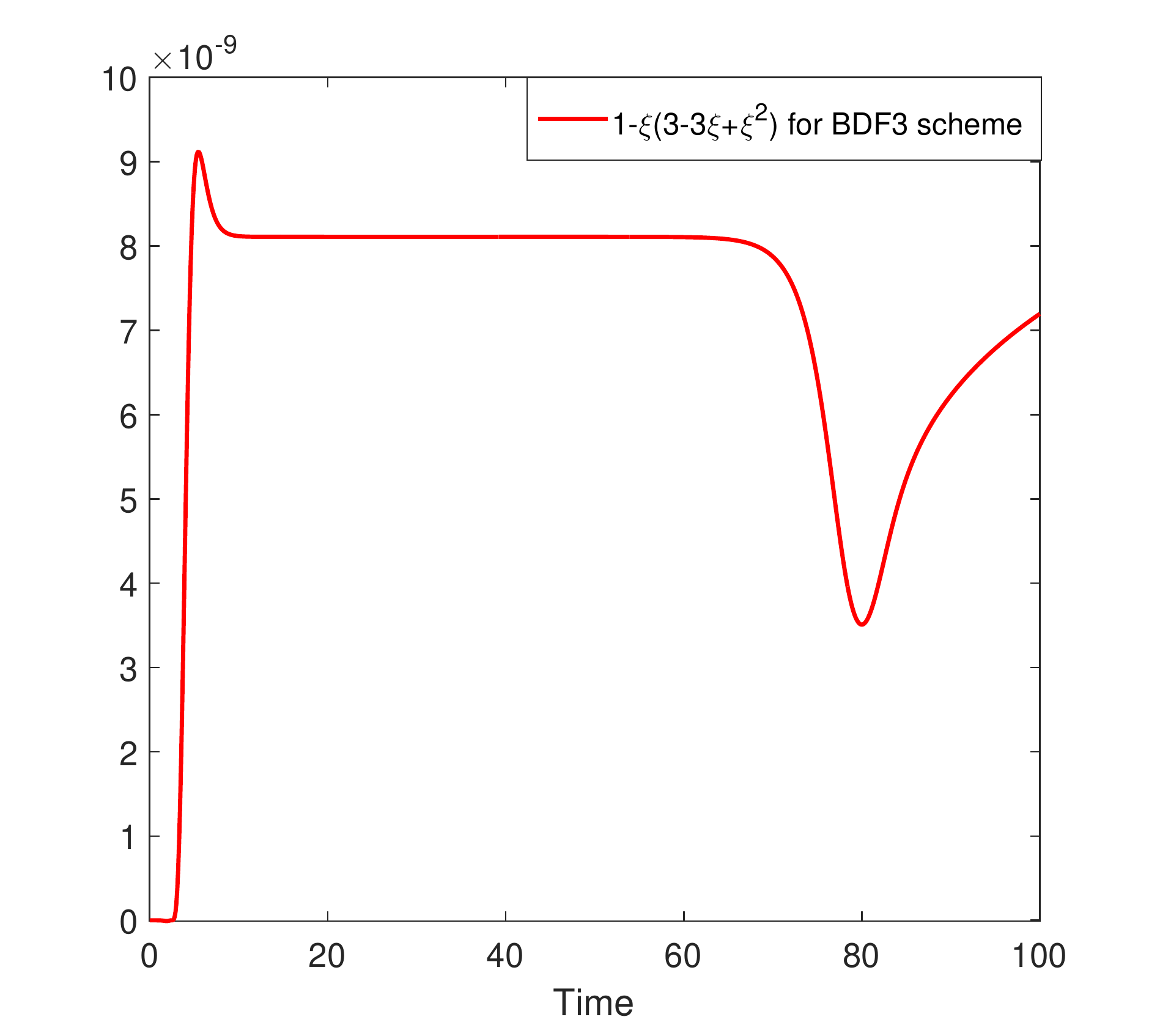}
\includegraphics[width=8cm,height=6cm]{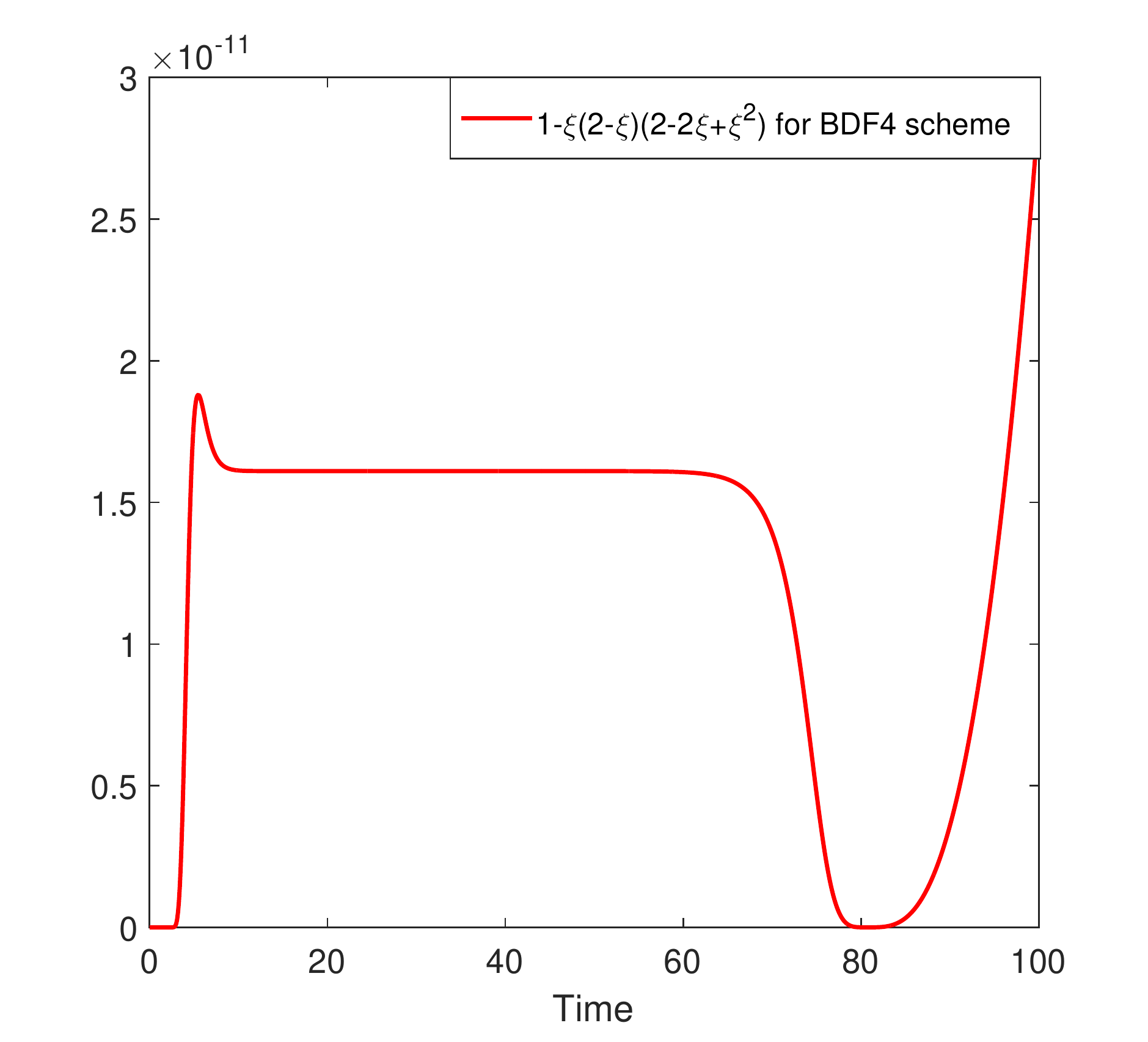}
\caption{Time evolution of $1-V(\xi)$ for the ESI-SAV CN, BDF2, BDF3 and BDF4 schemes with $\Delta t=0.1$.}\label{fig:fig1}
\end{figure}

\begin{figure}[htp]
\centering
\includegraphics[width=8cm,height=8cm]{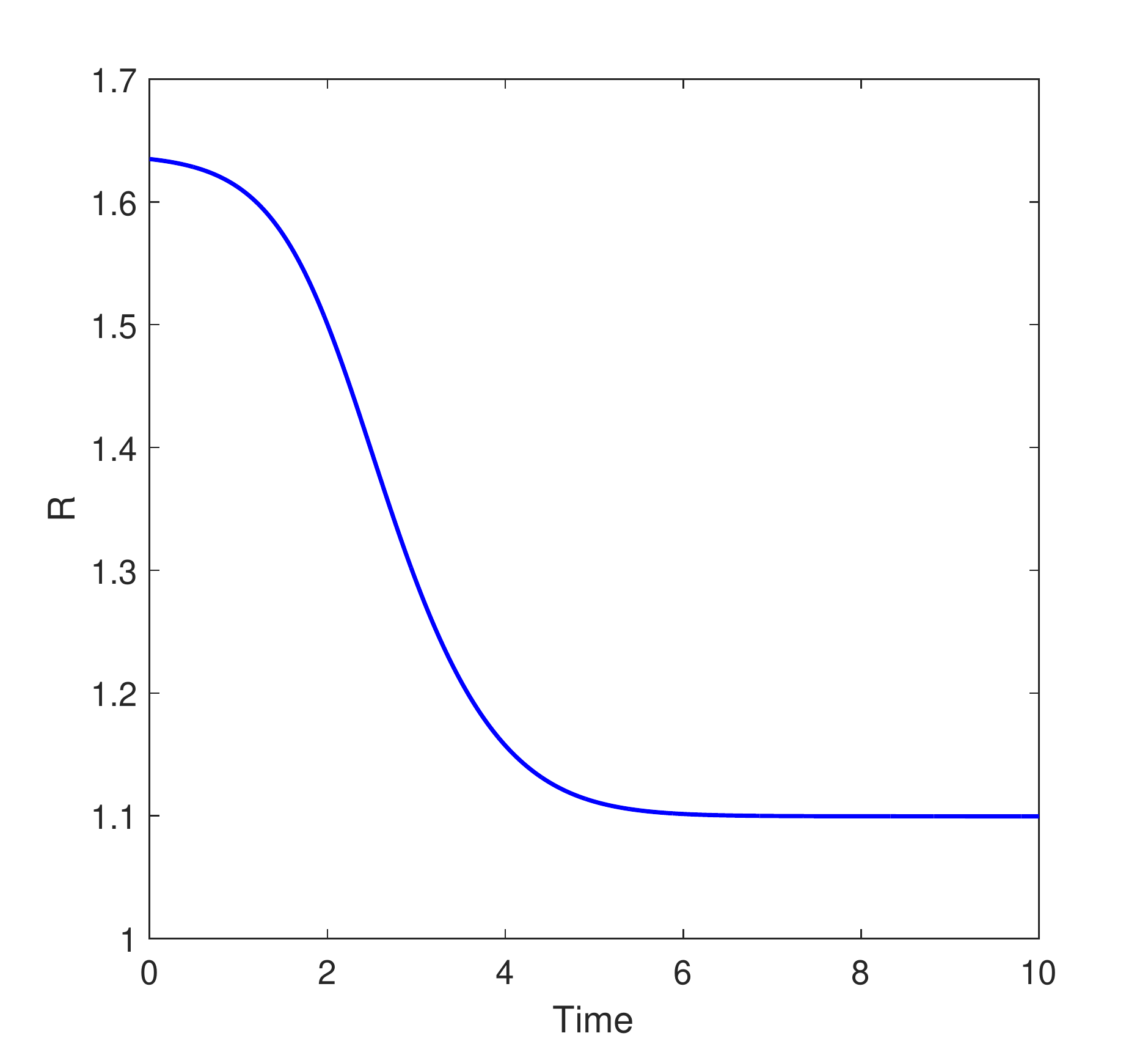}
\includegraphics[width=8cm,height=8cm]{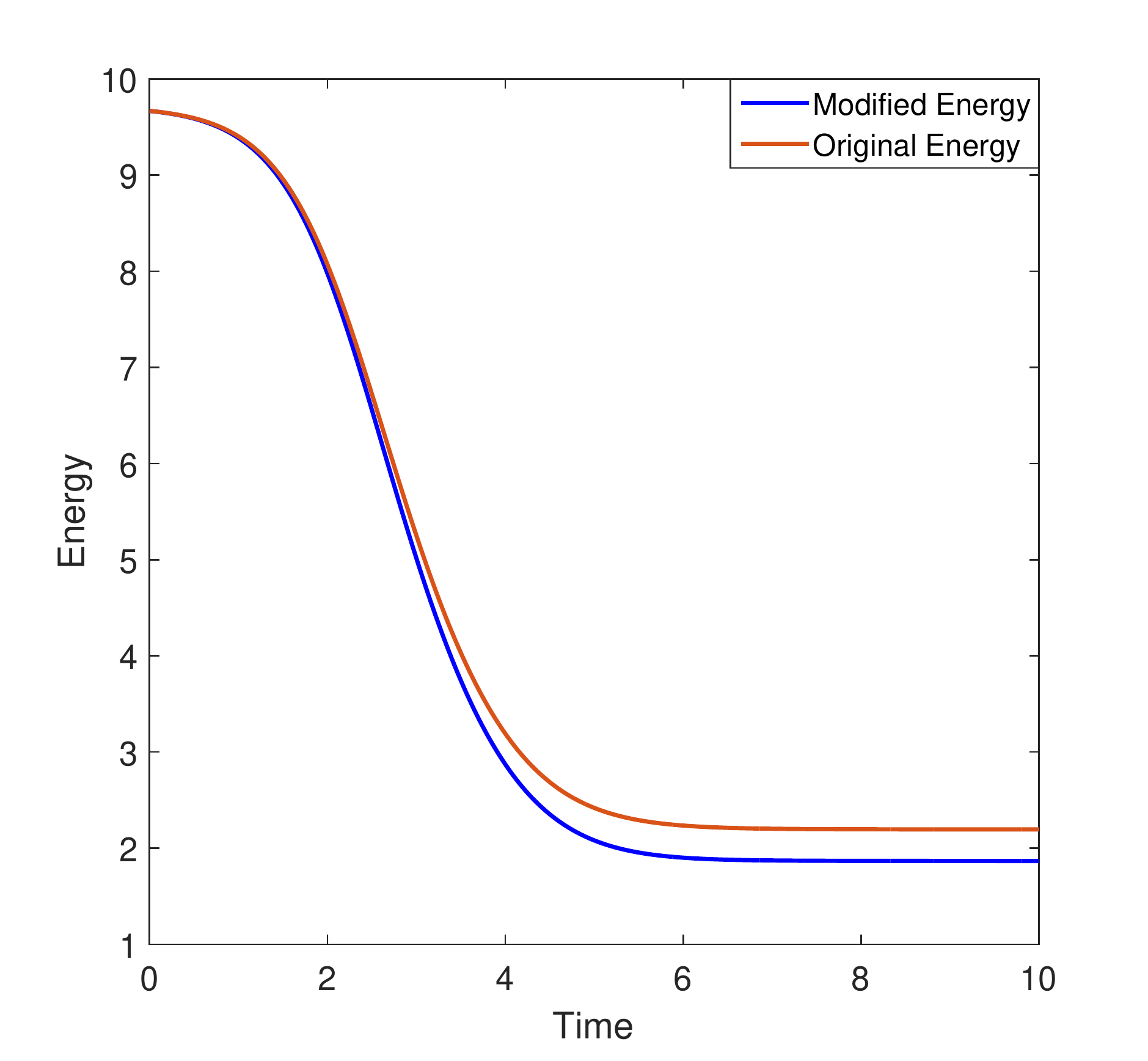}
\caption{Time evolution of both modified and original energy functional for $\Delta t=0.01$.}\label{fig:energy}
\end{figure}

\textbf{Example 2}: In the following, we solve a benchmark problem for the Cahn-Hilliard equation on $[0,2\pi)^2$ which can also be seen in many articles such as \cite{shen2018scalar}. When using the SAV type approach to simulate Cahn-Hilliard model, we need to specify the operators $\mathcal{L}=-\epsilon^2\Delta+\beta$ and $F(\phi)=\frac{1}{4}(\phi^2-1-\beta)^2$ to obtain stable simulation \cite{ShenA}. We take $\epsilon=0.025$, $\beta=2$ and discretize the space by the Fourier spectral method with $256\times256$ modes. The initial condition is chosen as the following
\begin{equation*}
\aligned
\phi_0(x,y,0)=0.25+0.4Rand(x,y),
\endaligned
\end{equation*}
where $Rand(x,y)$ is a randomly generated function.

Snapshots of the phase variable $\phi$ taken at $t=10$, $50$, $100$, $200$, $300$, $400$, $800$ and $1000$ with $\Delta t=0.1$ are shown in Figure \ref{fig:fig2}. The phase separation and coarsening process can be observed very simply which is consistent with the results in \cite{ShenA}. In Figure \ref{fig:fig3}, we plot the time evolution of the energy functional with five different time step size of $\Delta t=0.01$, $0.1$, $0.5$, $1$ and $5$ by using the first-order scheme based on the ESI-SAV approach. All energy curves show the monotonic decays for all time steps that confirms that the algorithm is unconditionally energy stable.
\begin{figure}[htp]
\centering
\subfigure[t=10]{
\includegraphics[width=3.8cm,height=3.8cm]{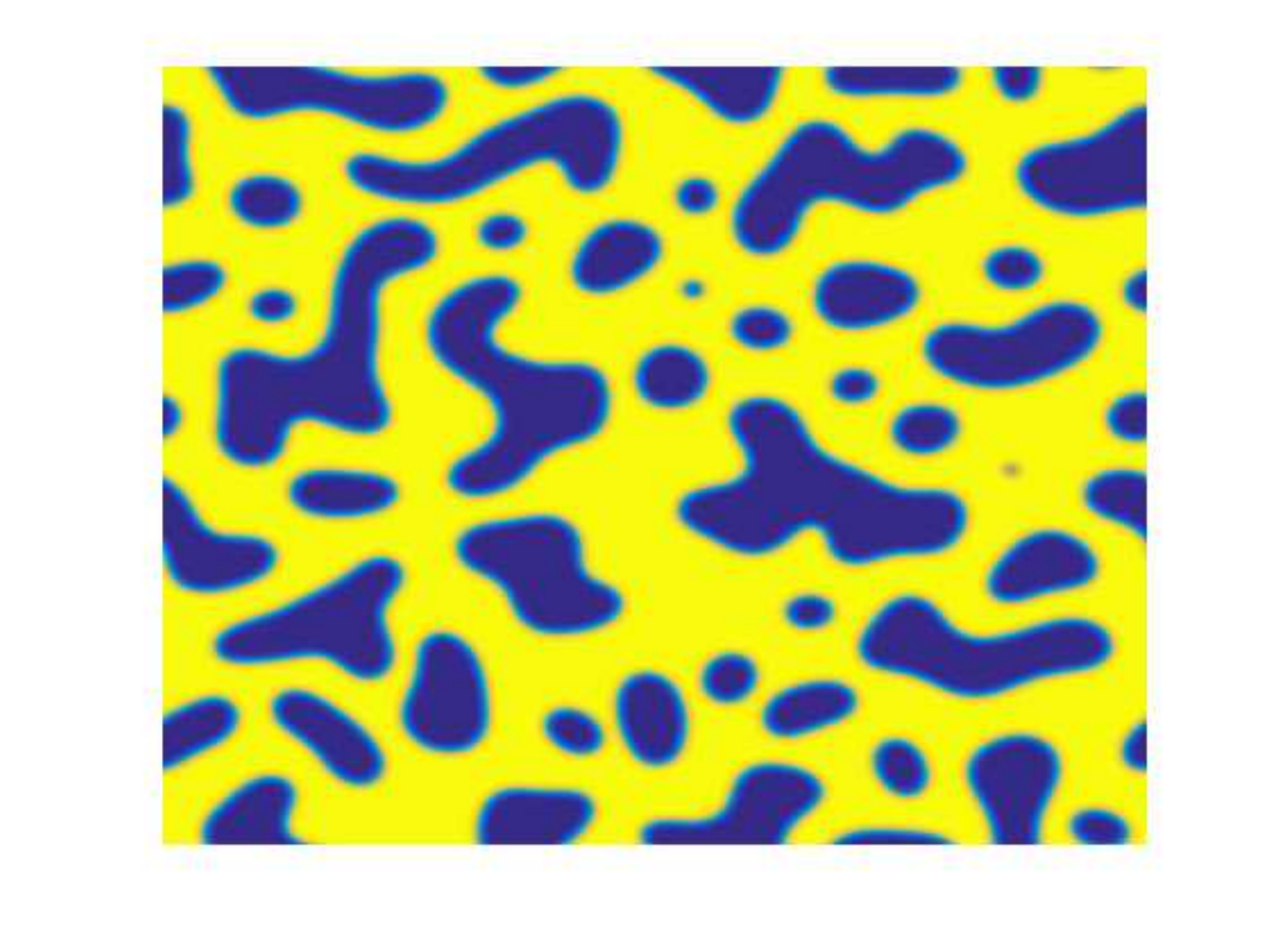}
\includegraphics[width=3.8cm,height=3.8cm]{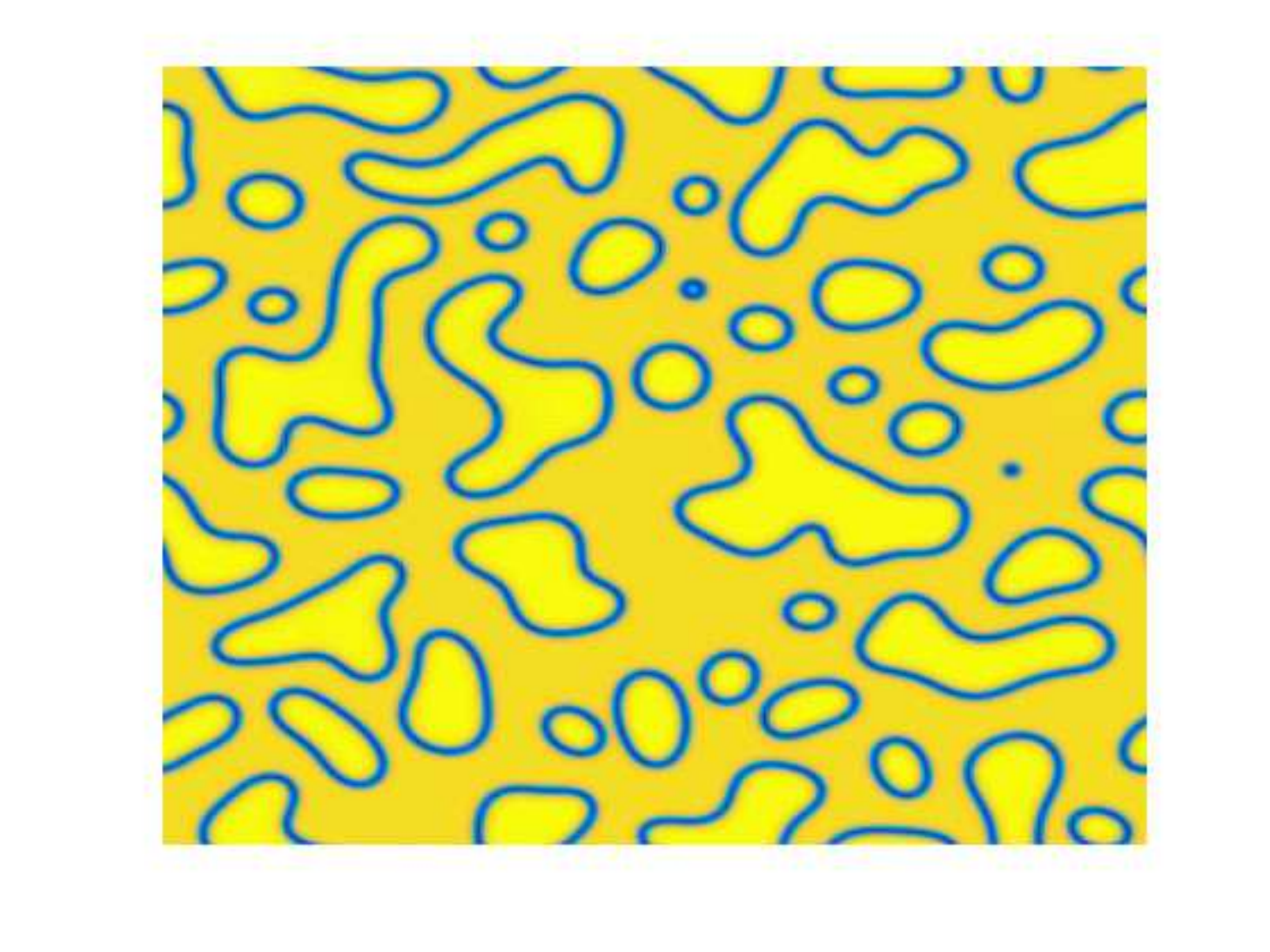}
}
\subfigure[t=50]
{
\includegraphics[width=3.8cm,height=3.8cm]{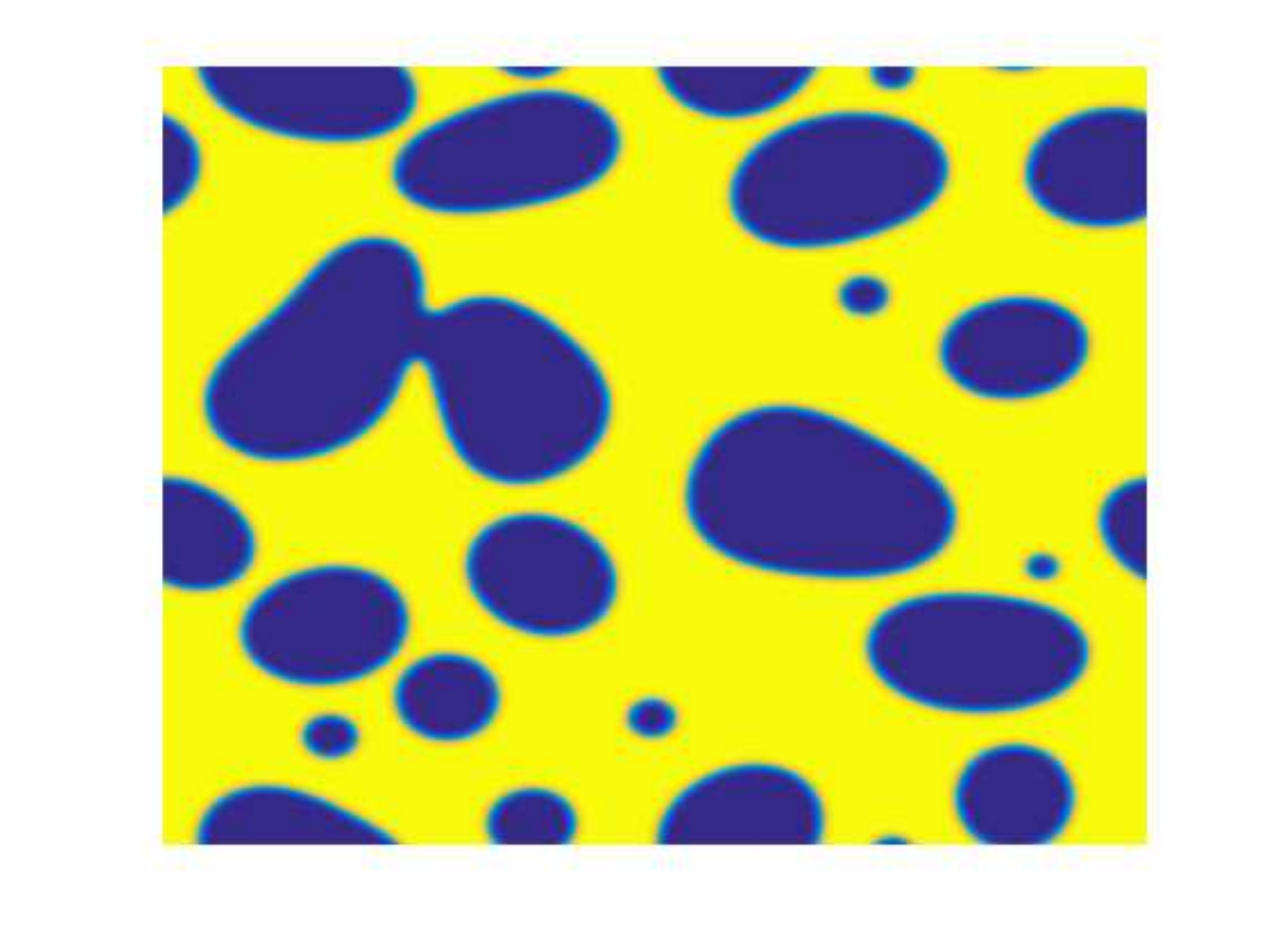}
\includegraphics[width=3.8cm,height=3.8cm]{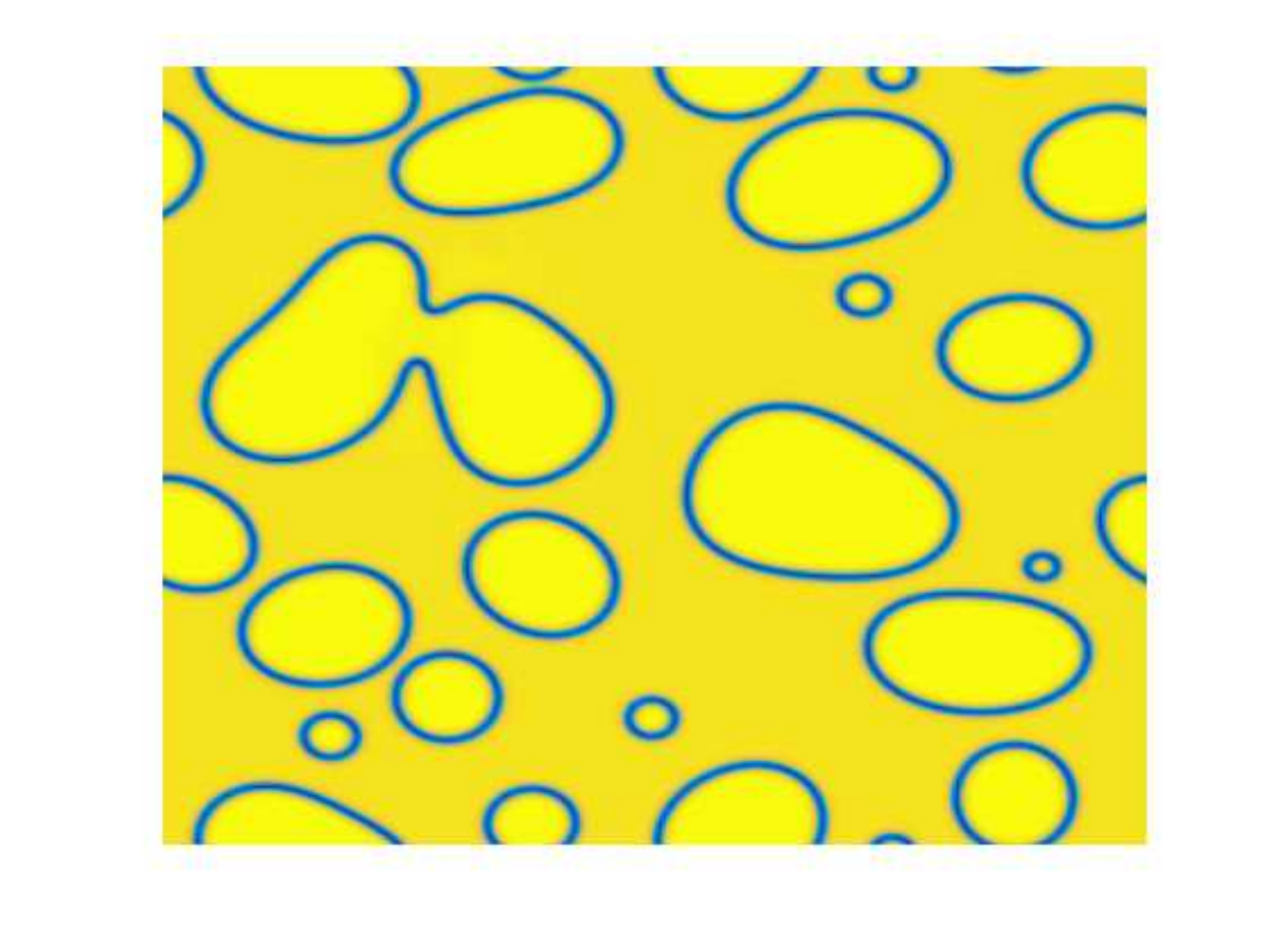}
}
\quad
\subfigure[t=100]
{
\includegraphics[width=3.8cm,height=3.8cm]{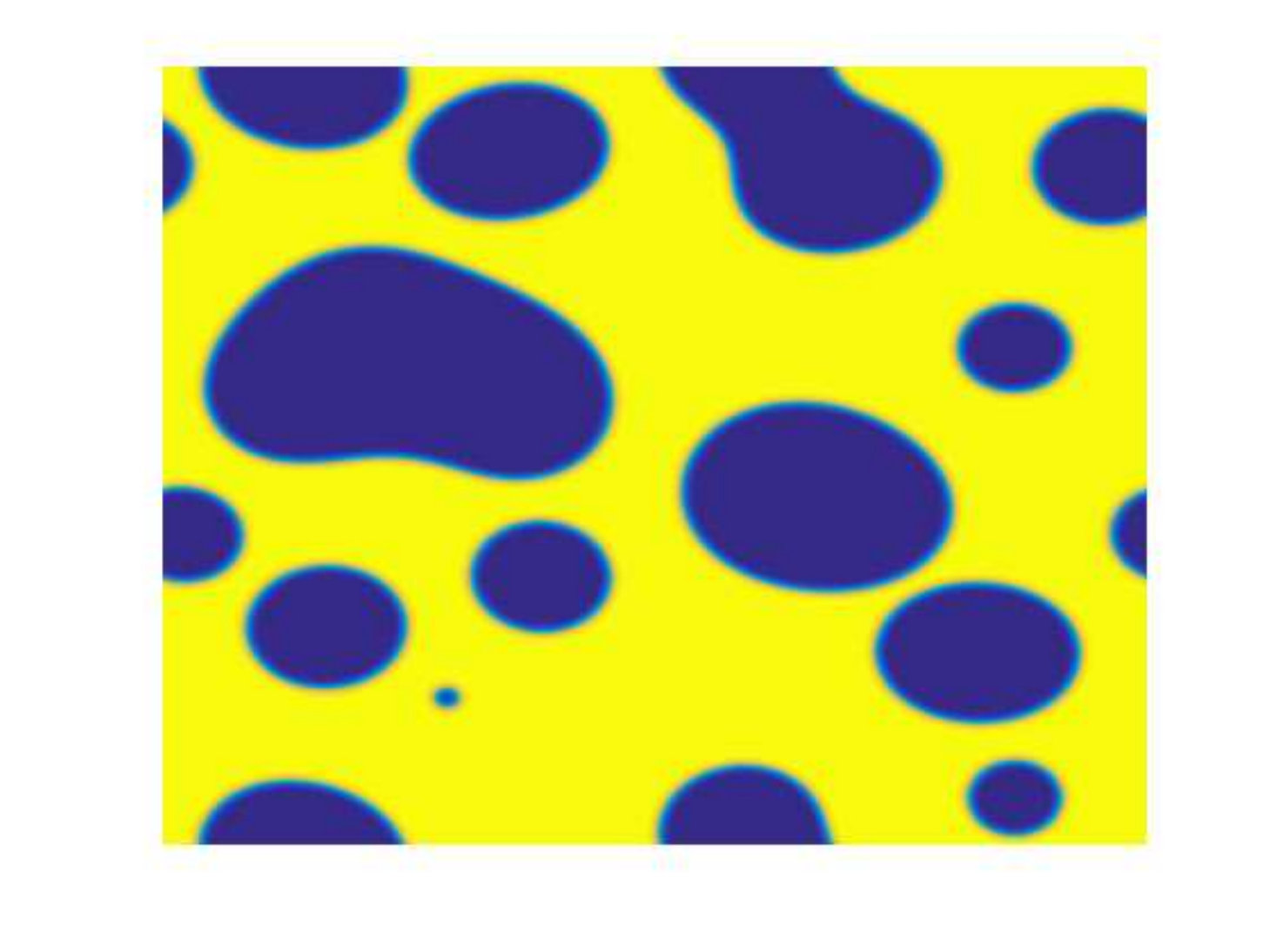}
\includegraphics[width=3.8cm,height=3.8cm]{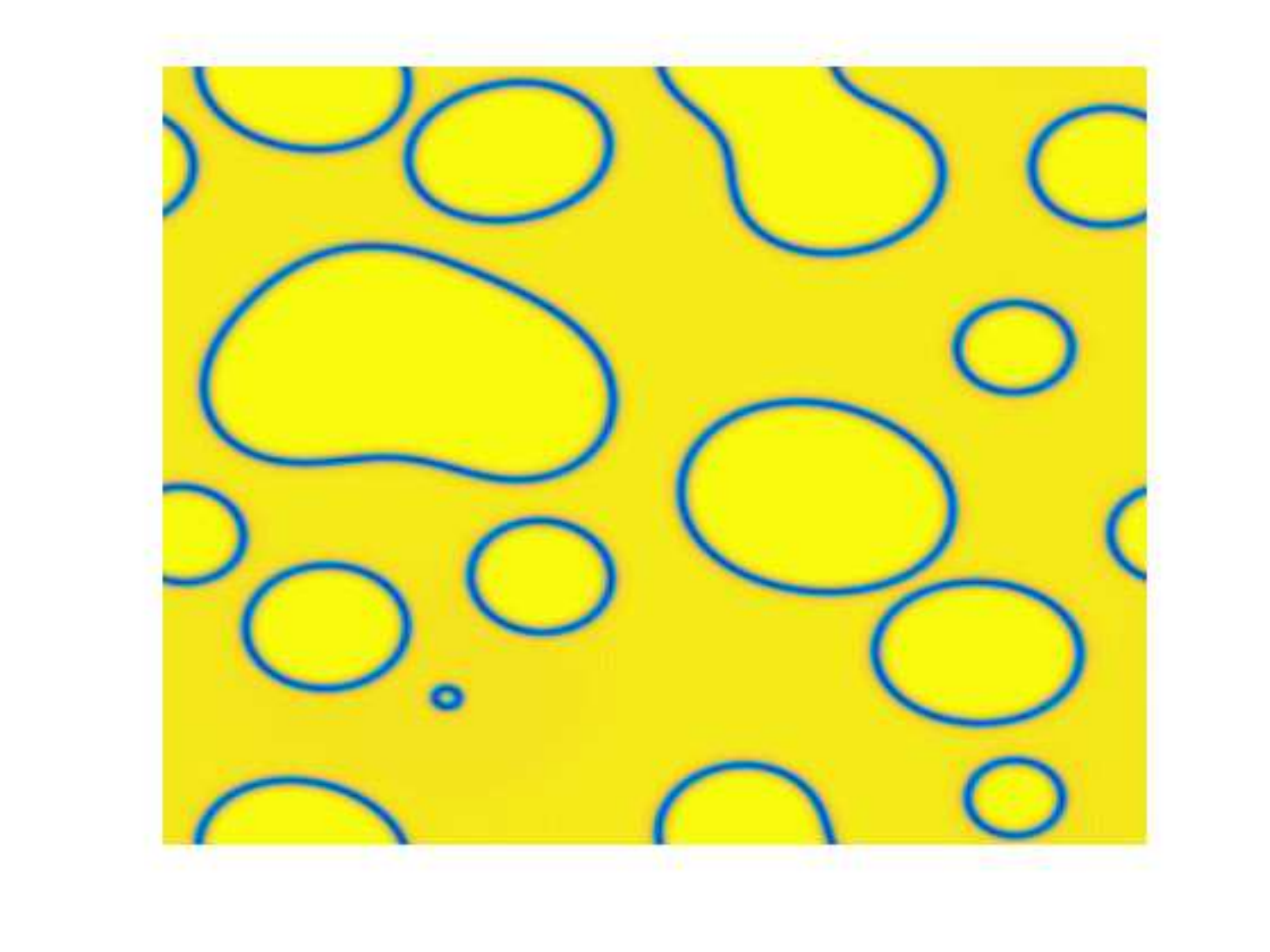}
}
\subfigure[t=200]
{
\includegraphics[width=3.8cm,height=3.8cm]{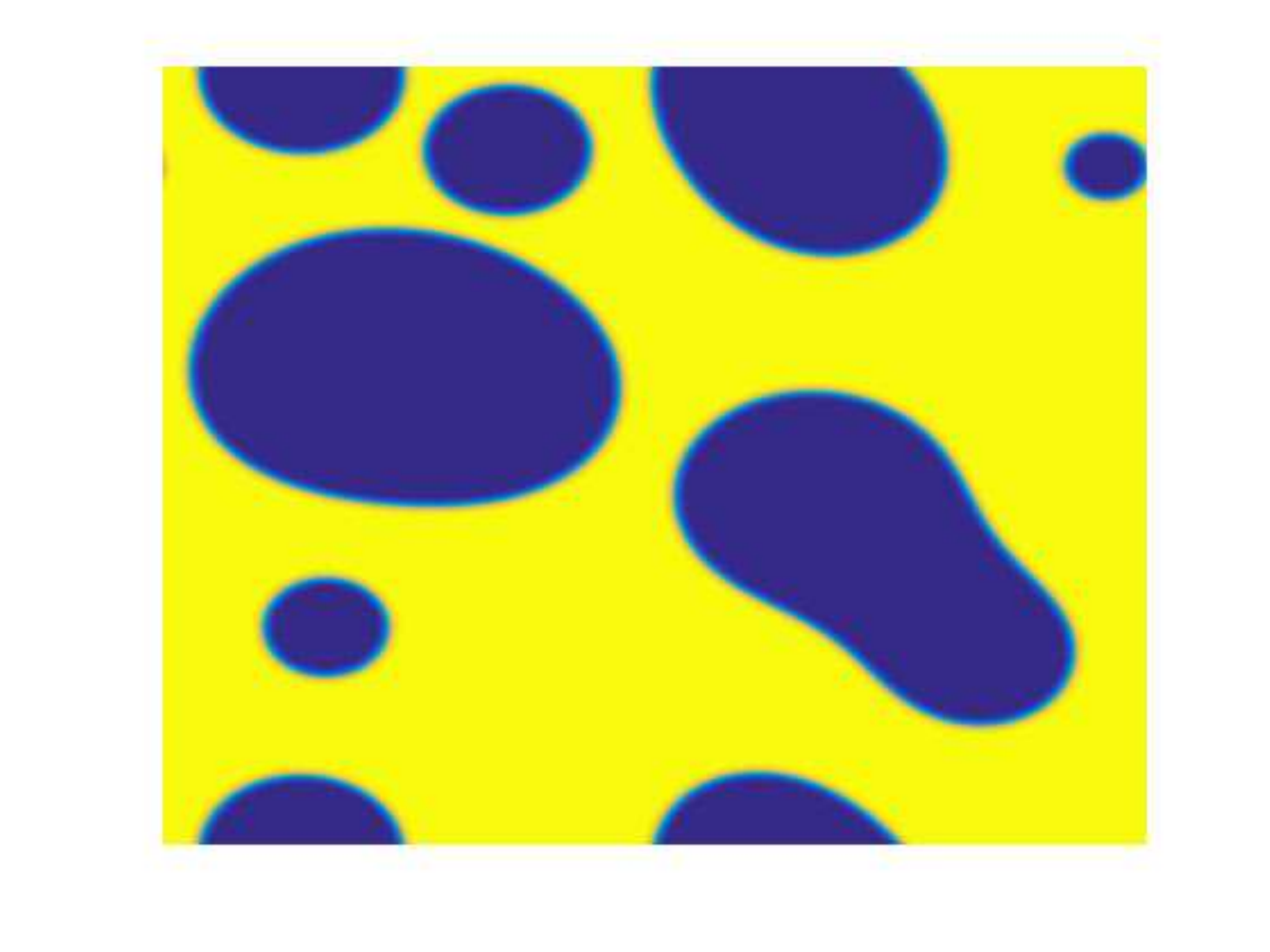}
\includegraphics[width=3.8cm,height=3.8cm]{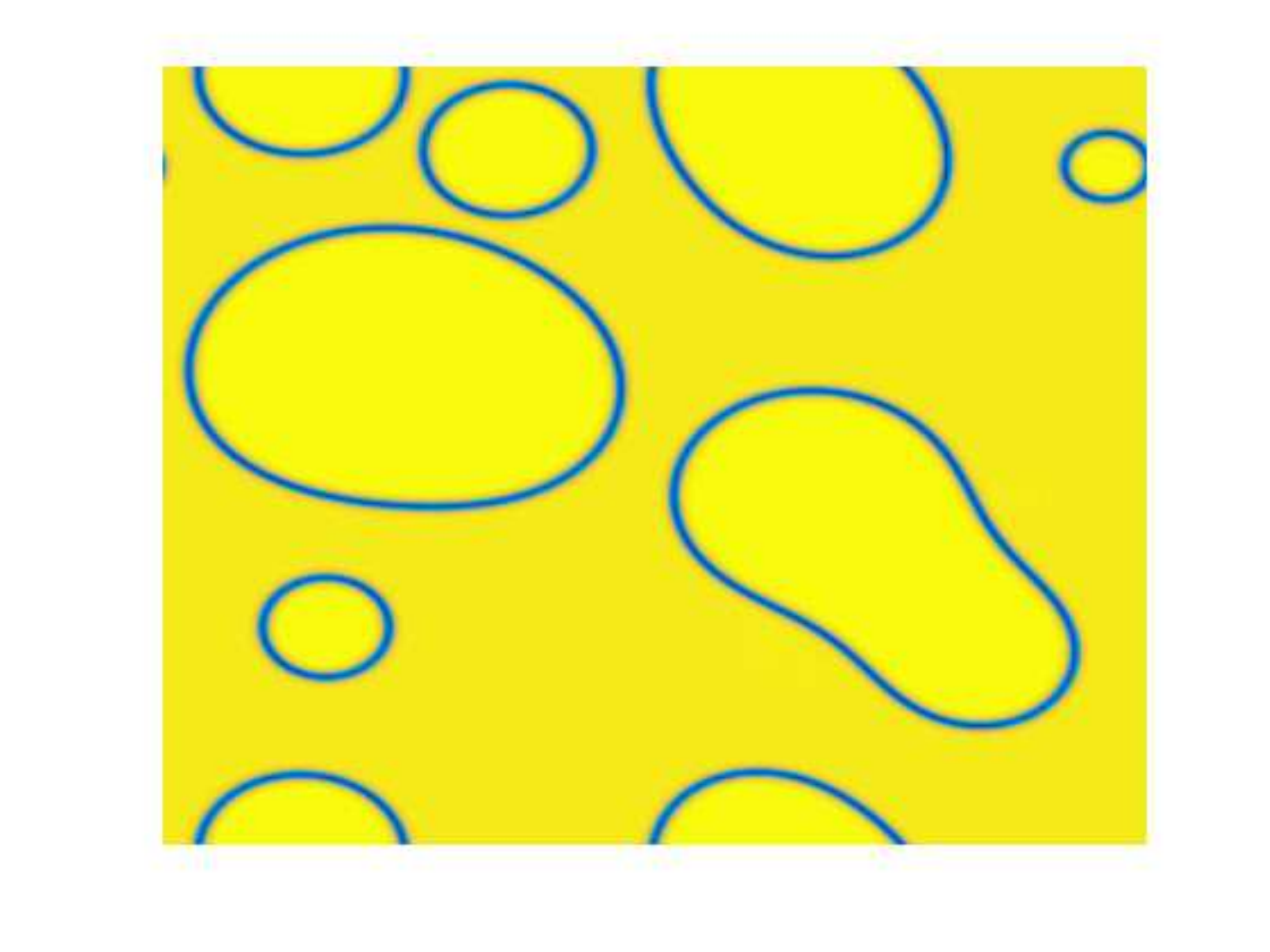}
}
\quad
\subfigure[t=400]
{
\includegraphics[width=3.8cm,height=3.8cm]{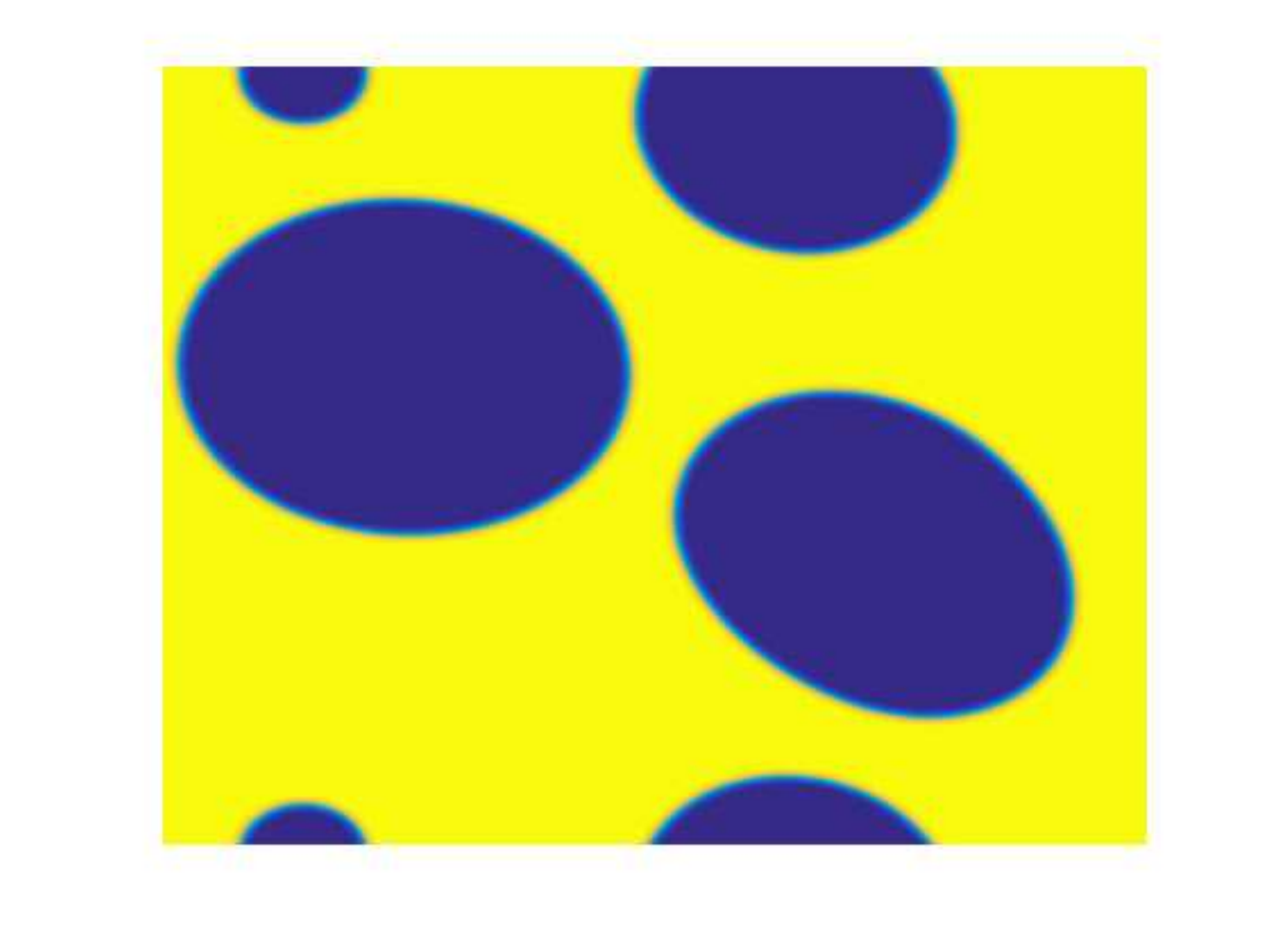}
\includegraphics[width=3.8cm,height=3.8cm]{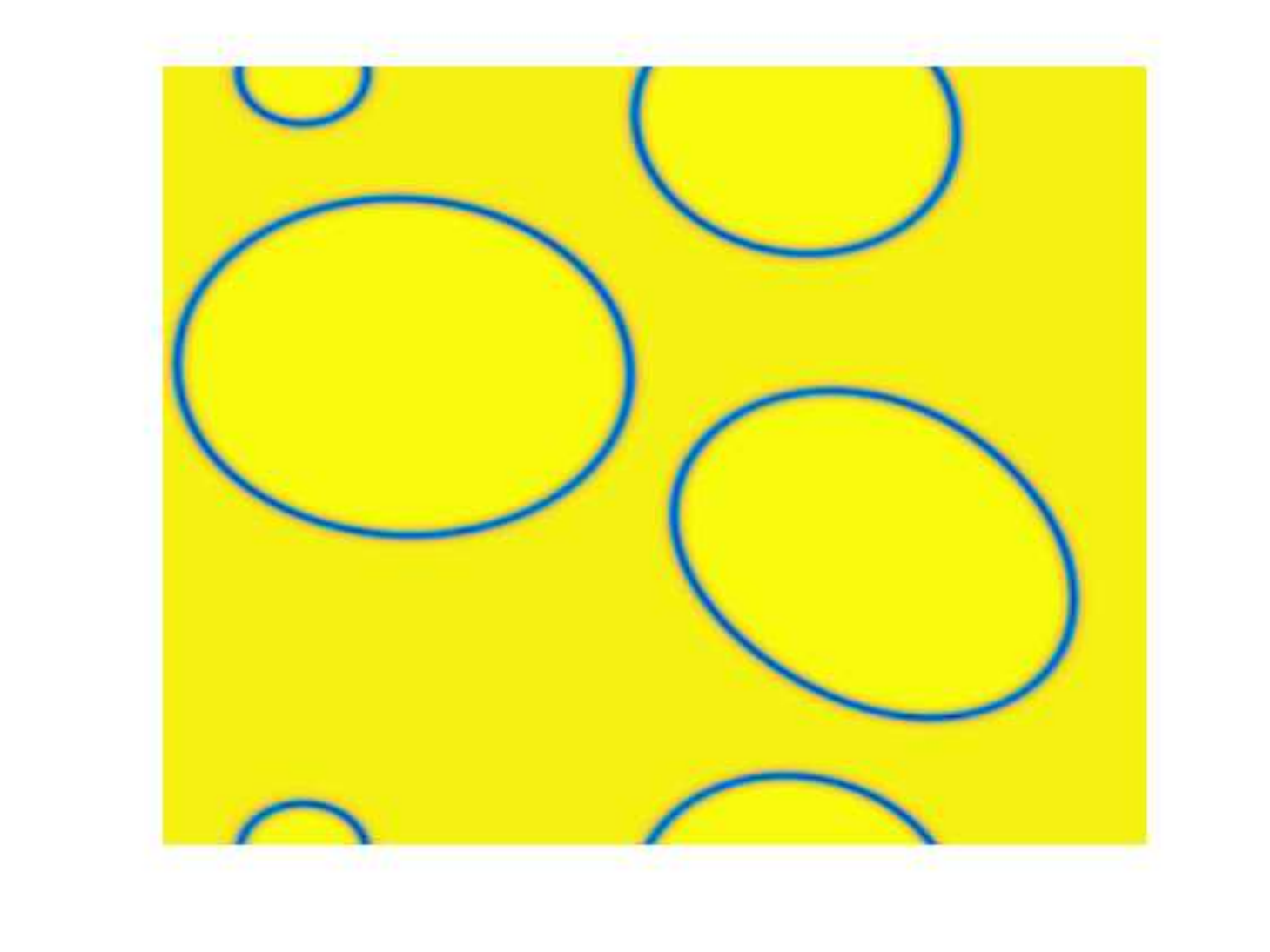}
}
\subfigure[t=1000]
{
\includegraphics[width=3.8cm,height=3.8cm]{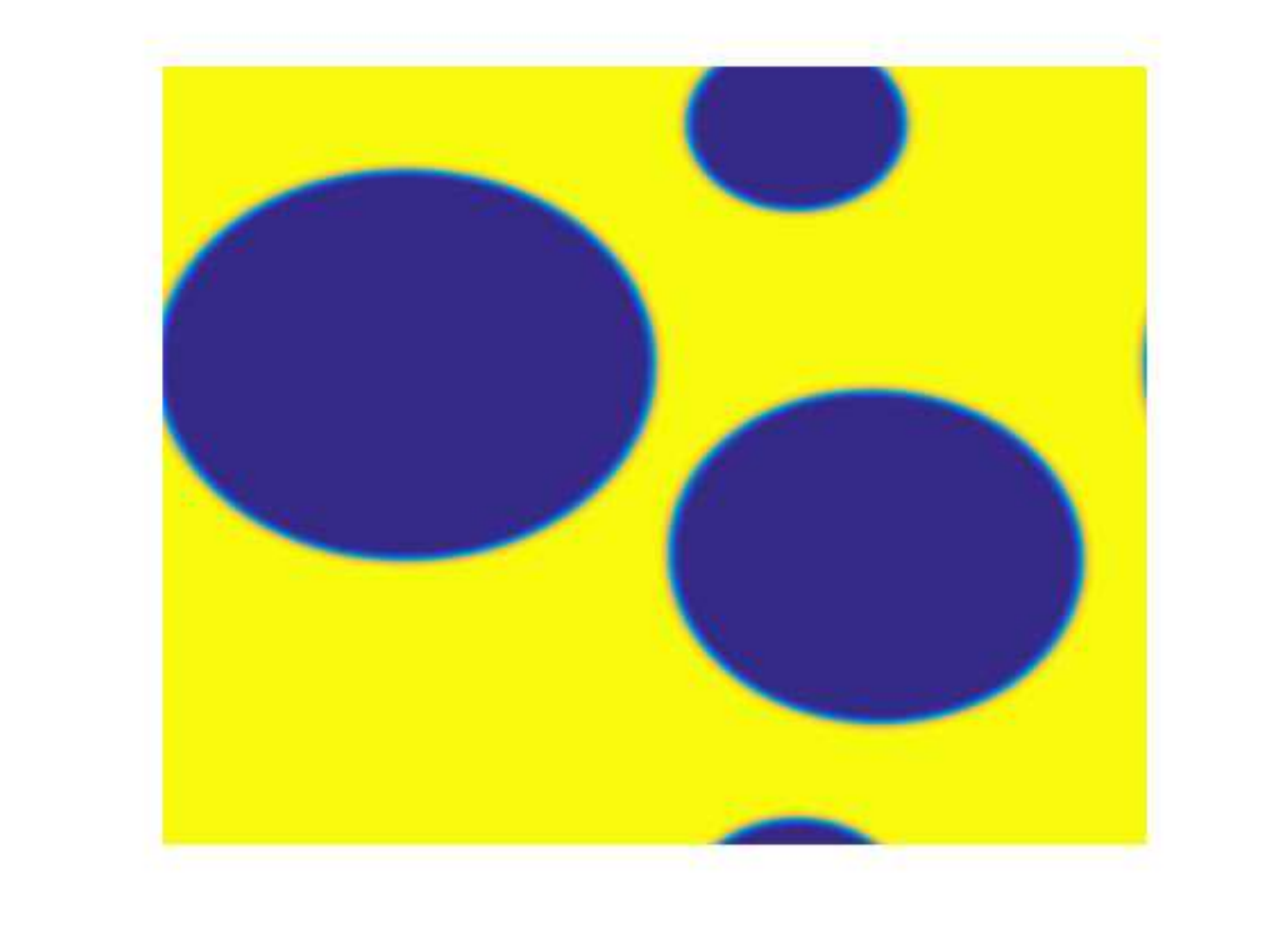}
\includegraphics[width=3.8cm,height=3.8cm]{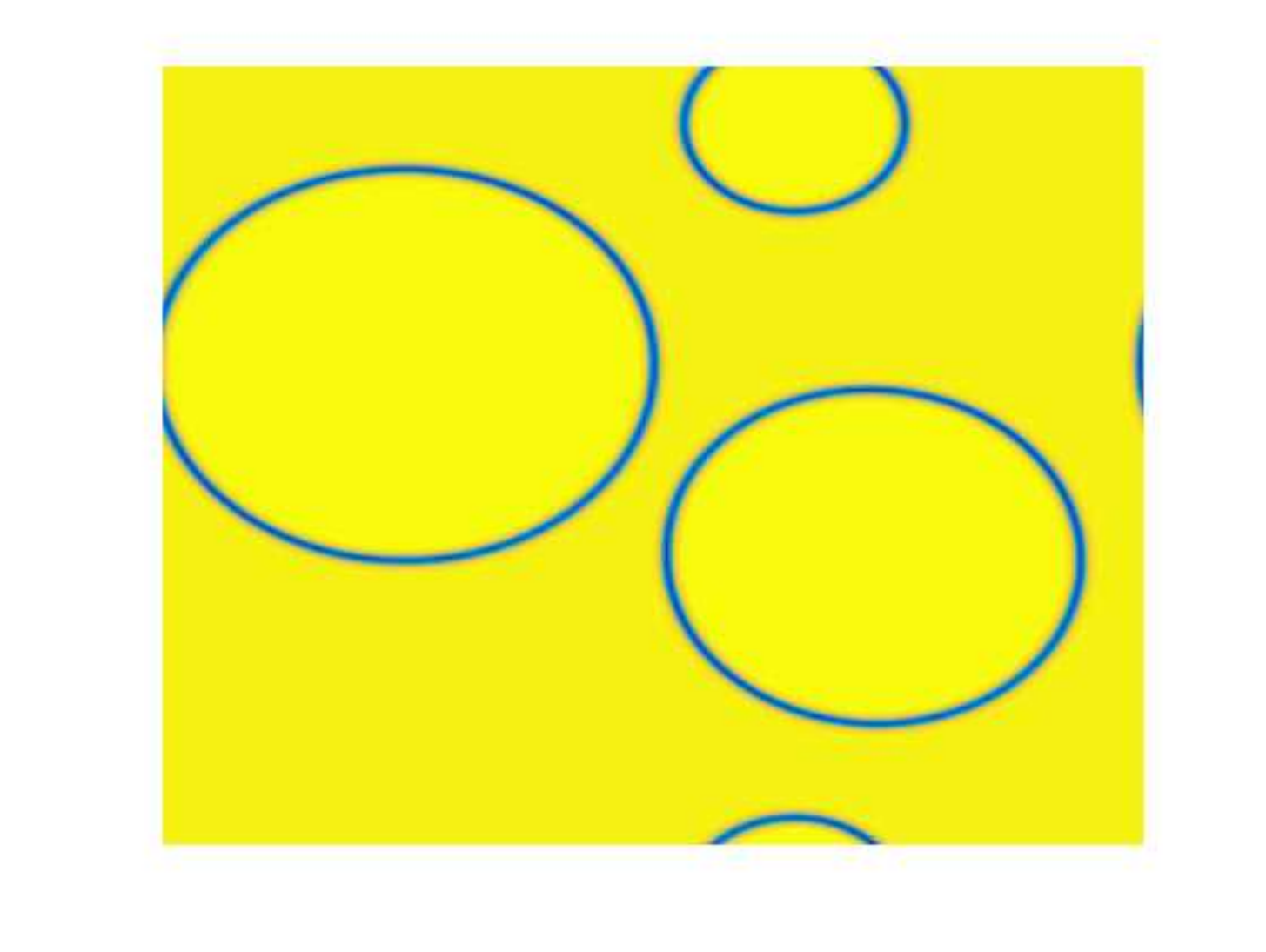}
}
\caption{Snapshots of the phase variable $\phi$ and phase interface are taken at t=10, 50, 100, 200, 400 and 1000 with $\Delta t=0.1$ for example 2.}\label{fig:fig2}
\end{figure}
\begin{figure}[htp]
\centering
\includegraphics[width=12cm,height=10cm]{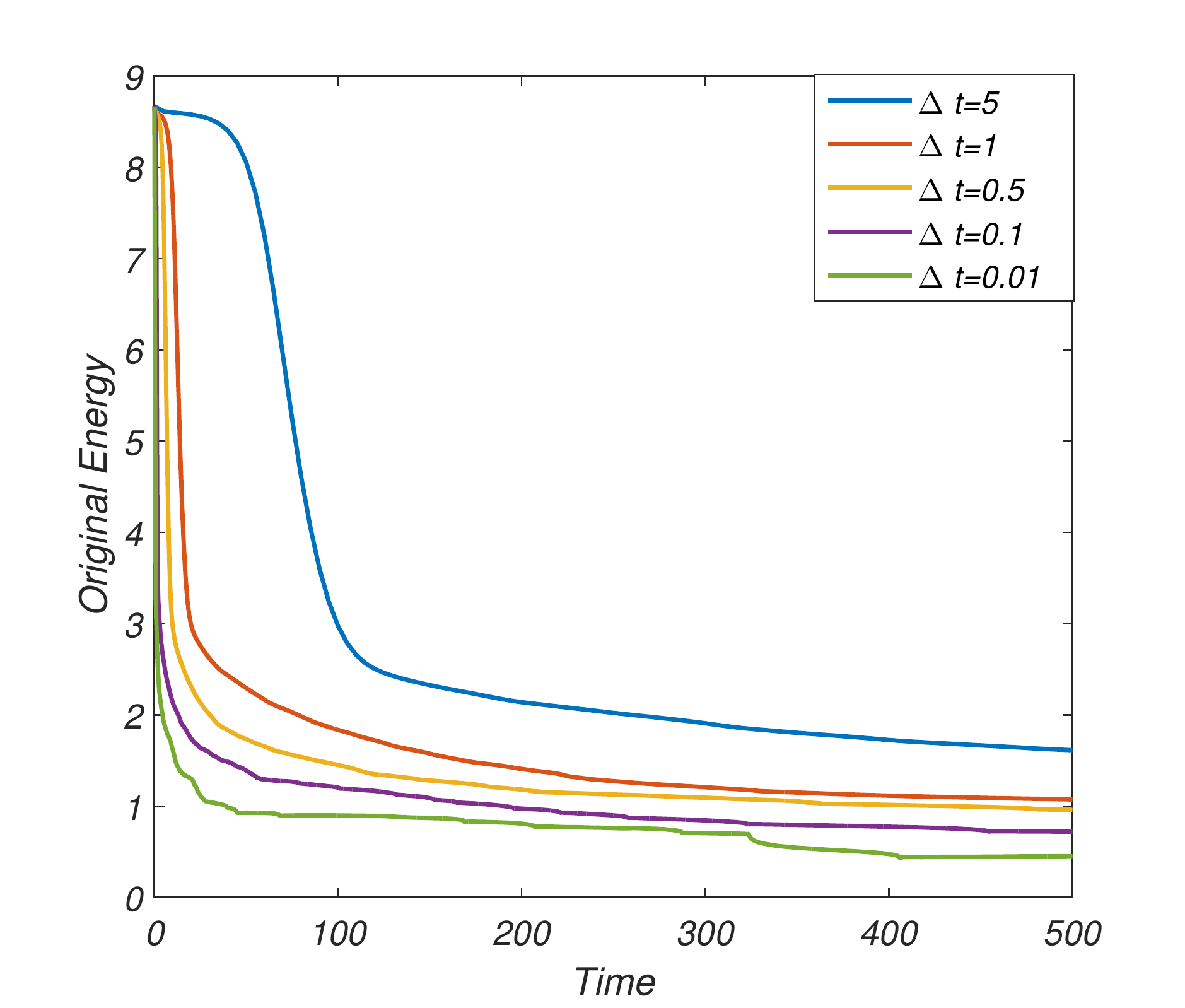}
\caption{Time evolution of the energy functional for five different time steps of $\Delta t=0.01$, $0.1$, $0.5$, $1$ and $5$.}\label{fig:fig3}
\end{figure}
\subsection{Phase field crystal equations}
In this subsection, we will simulate the phase transition behavior of the phase field crystal model for the proposed ES-SAV approach. The similar numerical example can be found in many articles such as \cite{li2017efficient,yang2017linearly}. Elder \cite{elder2002modeling} firstly proposed the phase field crystal (PFC) model based on density functional theory in 2002. This model can simulate the evolution of crystalline microstructure on atomistic length and diffusive time scales. It naturally incorporates elastic and plastic deformations and multiple crystal orientations, and can be applied to many different physical phenomena.

In particular, consider the following Swift-Hohenberg free energy:
\begin{equation*}
E(\phi)=\int_{\Omega}\left(\frac{1}{4}\phi^4+\frac{1}{2}\phi\left(-\epsilon+(1+\Delta)^2\right)\phi\right)d\textbf{x},
\end{equation*}
where $\textbf{x} \in \Omega \subseteq \mathbb{R}^d$, $\phi$ is the density field and $\epsilon$ is a positive bifurcation constant with physical significance. $\Delta$ is the Laplacian operator.

Considering a gradient flow in $H^{-1}$, one can obtain the phase field crystal equation under the constraint of mass conservation as follows:
\begin{equation*}
\frac{\partial \phi}{\partial t}=\Delta\mu=\Delta\left(\phi^3-\epsilon\phi+(1+\Delta)^2\phi\right), \quad(\textbf{x},t)\in\Omega\times Q,
\end{equation*}
which is a sixth-order nonlinear parabolic equation and can be applied to simulate various phenomena such as crystal growth, material hardness and phase transition. Here $Q=(0,T]$, $\mu=\frac{\delta E}{\delta \phi}$ is called the chemical potential.

\textbf{Example 3}: Consider the PFC model in the computational domain $\Omega=[0,128]^2$ with the following initial condition
\begin{equation*}
\aligned
&\phi_0(x,y)=0.07+0.07Rand(x,y),
\endaligned
\end{equation*}
where the $Rand(x,y)$ is the random number in $[-1,1]$ with zero mean. The order parameter is $\epsilon=0.025$. we set $256^2$ Fourier modes to discretize the two dimensional space.

We show the phase transition behavior of the density field at various times in Figure \ref{fig:fig4} with both SAV and ESI-SAV methods. Configuration evolutions for PFC model by SAV and ESI-SAV schemes are taken at $t=40$, $400$, $1000$, and $3000$. No visible difference is observed which means the new proposed method is effective.

\begin{figure}[htp]
\centering
\begin{tabular}{ccccc}
SAV&\includegraphics[width=3.6cm,height=3.6cm]{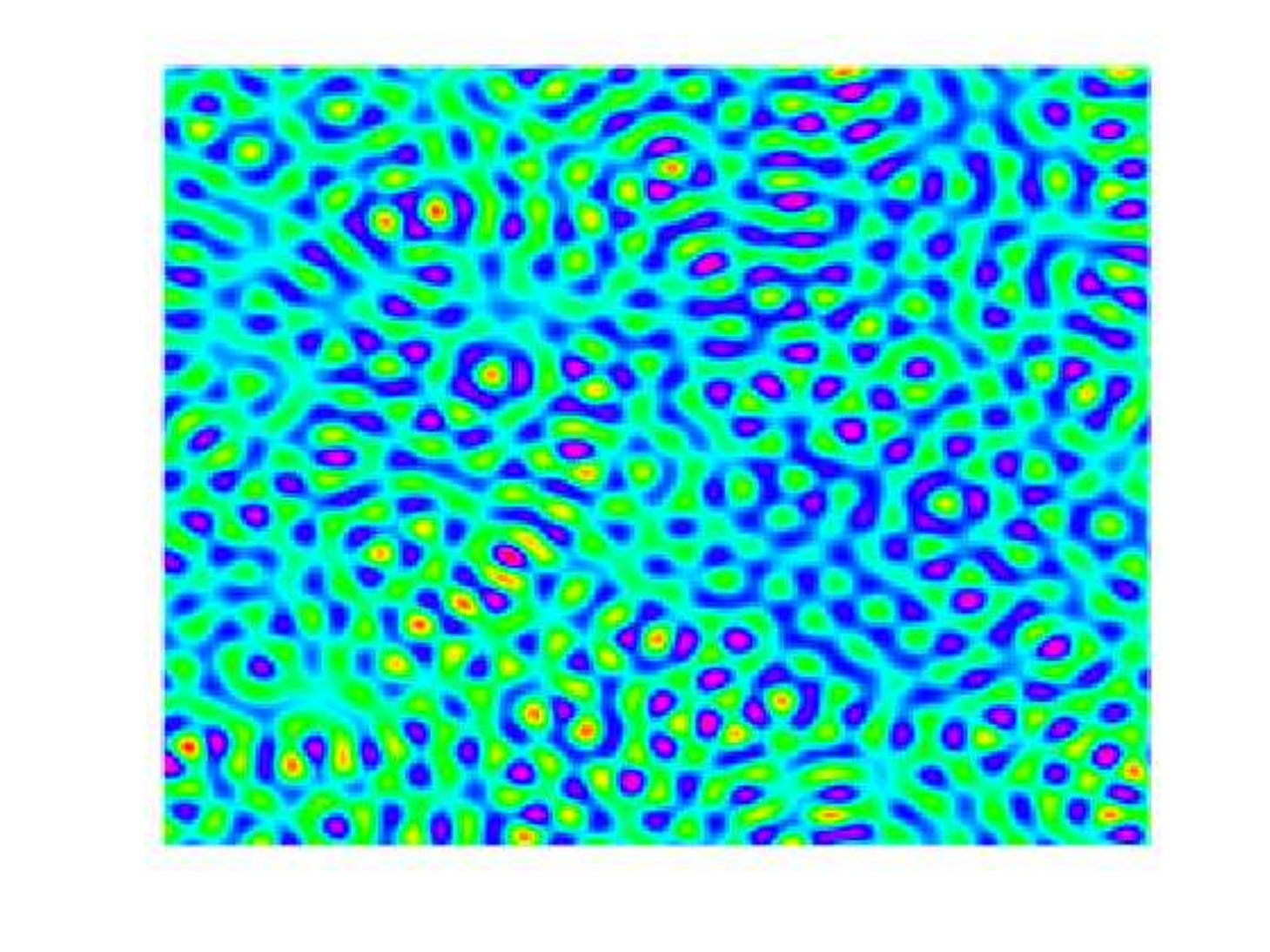}
\includegraphics[width=3.6cm,height=3.6cm]{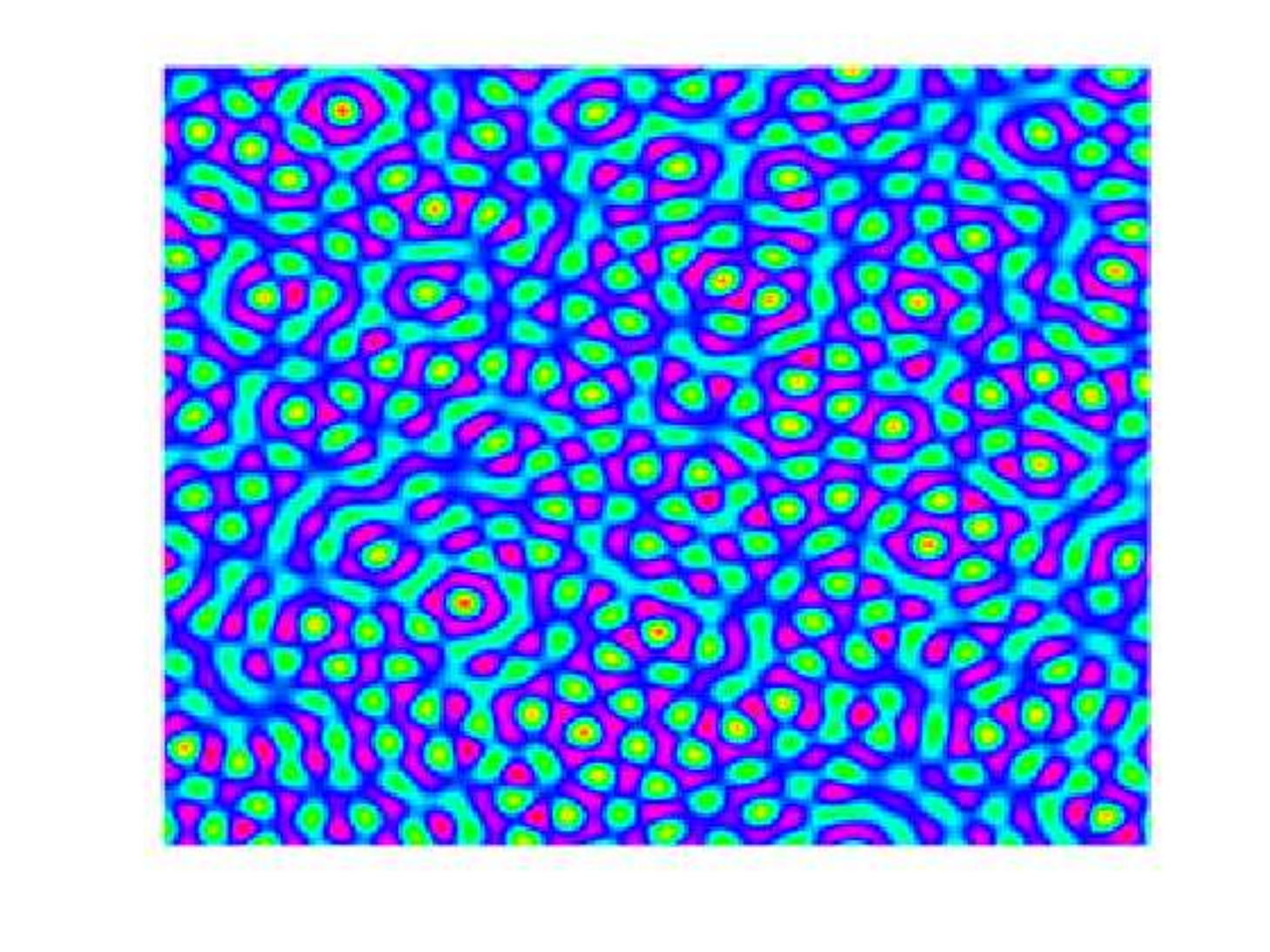}
\includegraphics[width=3.6cm,height=3.6cm]{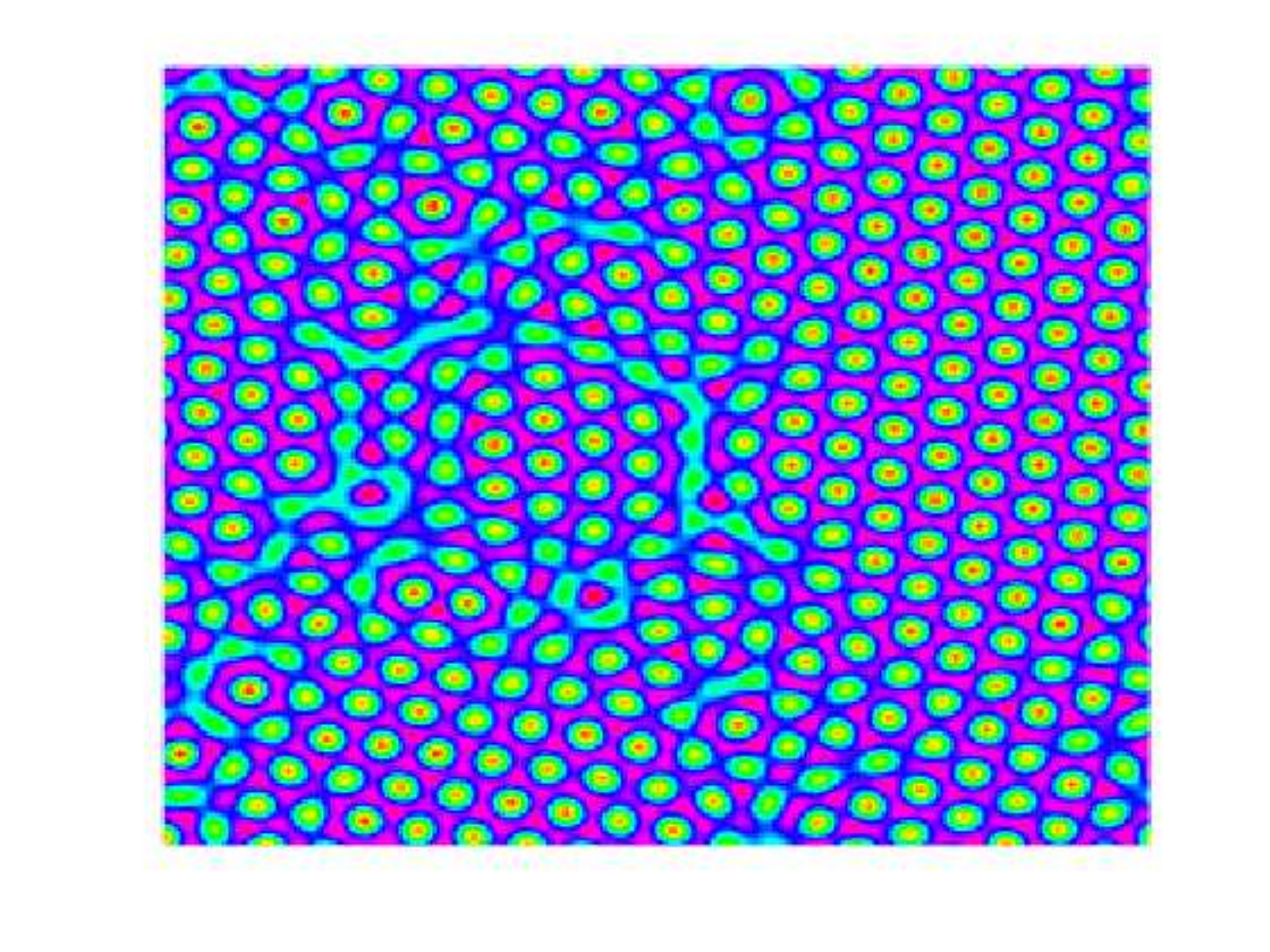}
\includegraphics[width=3.6cm,height=3.6cm]{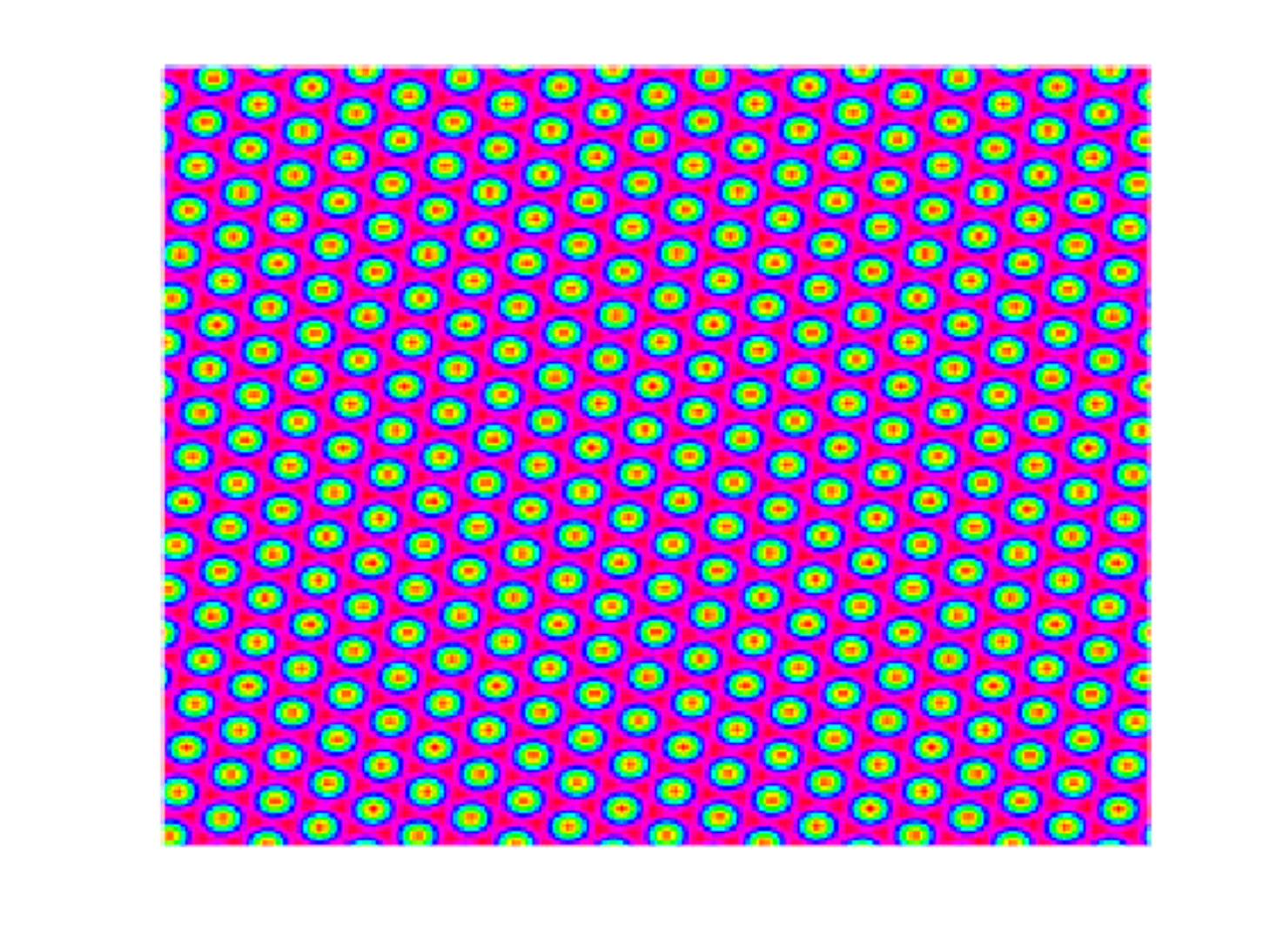}\\
\Xhline{1.2pt}
ESI-SAV&\includegraphics[width=3.6cm,height=3.6cm]{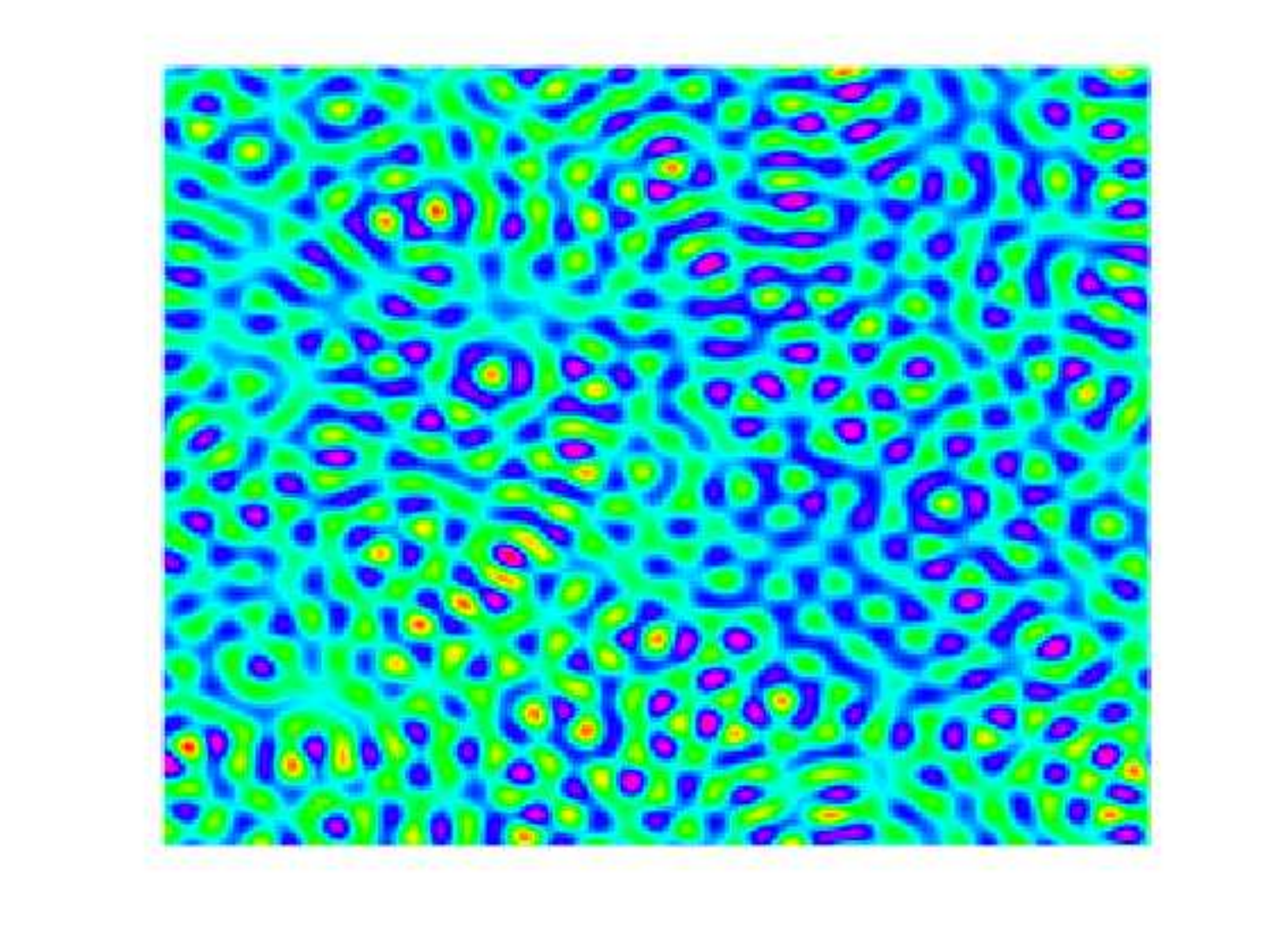}
\includegraphics[width=3.6cm,height=3.6cm]{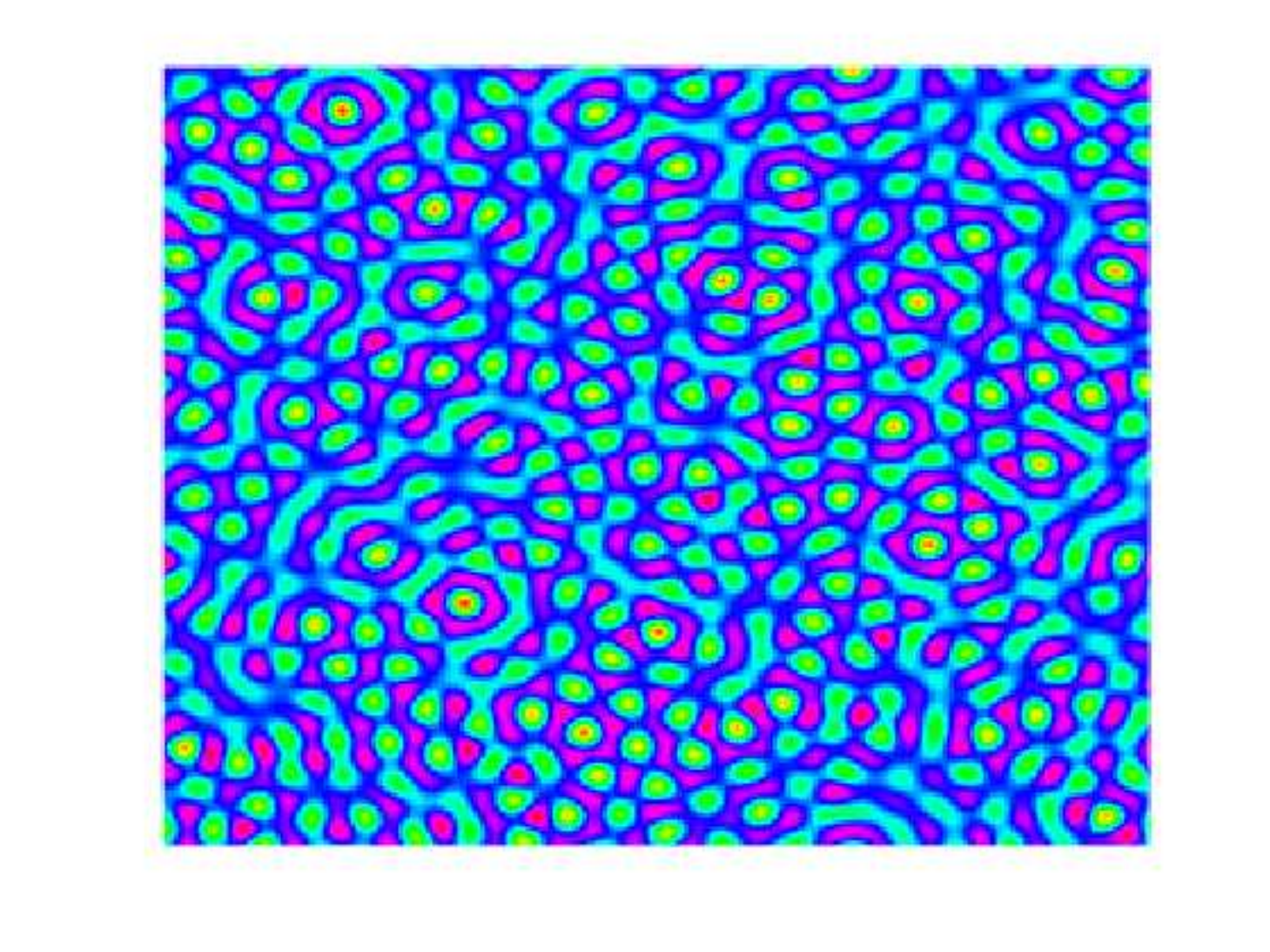}
\includegraphics[width=3.6cm,height=3.6cm]{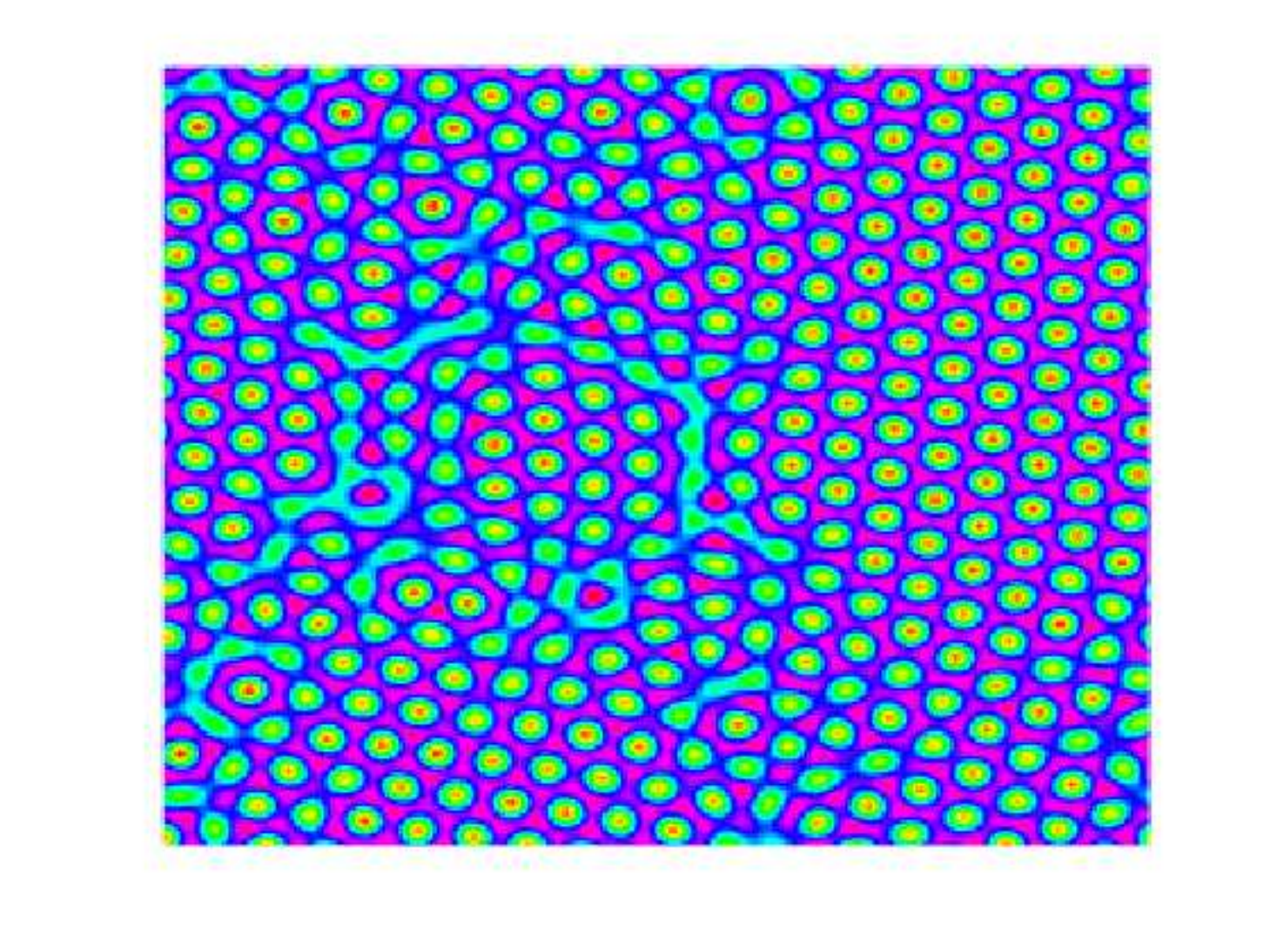}
\includegraphics[width=3.6cm,height=3.6cm]{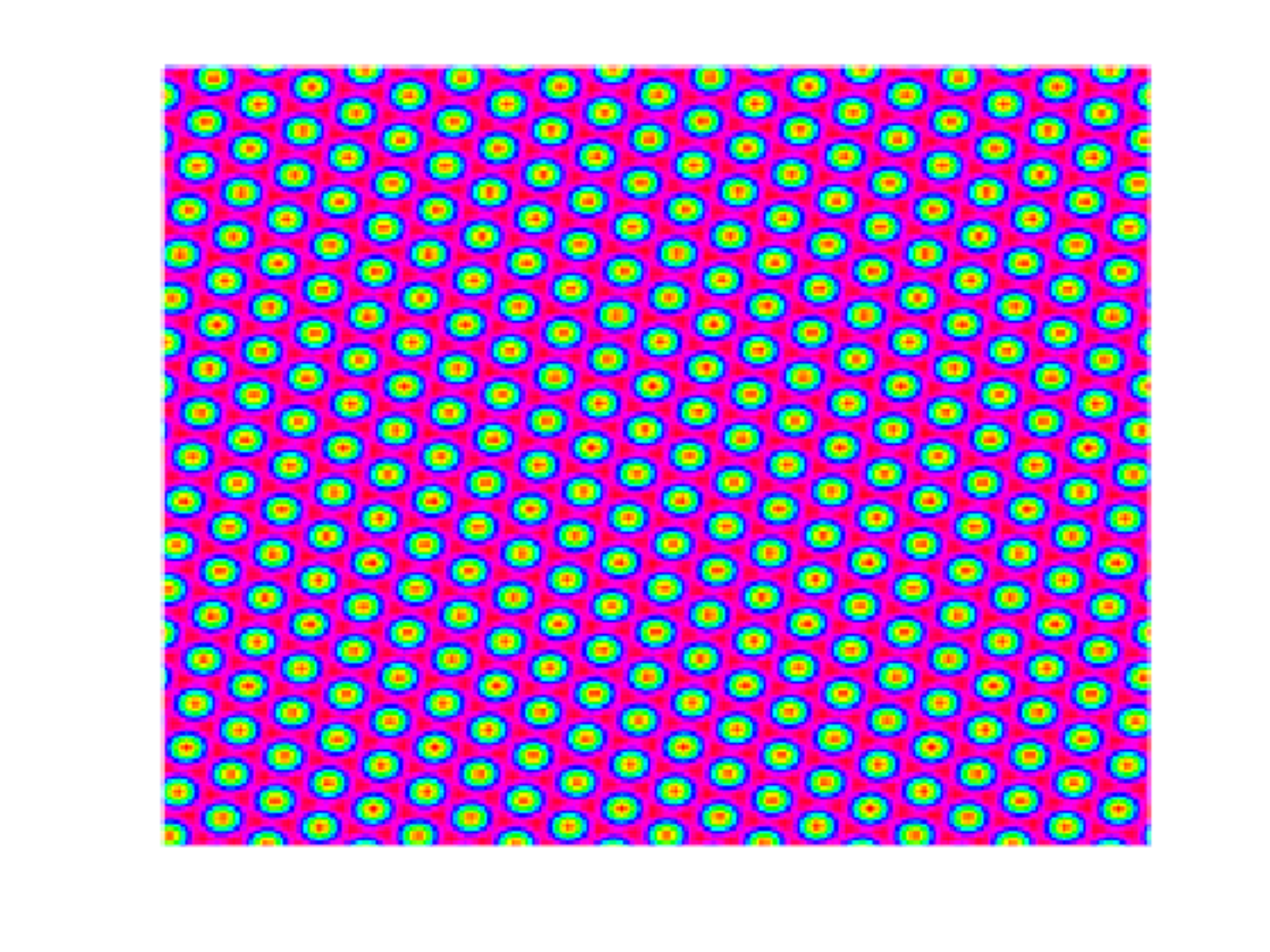}\\
\Xhline{1.2pt}
\end{tabular}
\caption{Configuration evolutions for PFC model by SAV and ESI-SAV schemes are taken at $t=40$, $400$, $1000$, and $3000$.}\label{fig:fig4}
\end{figure}

\textbf{Example 4}: The process of crystallization in a supercool liquid is very classical example.  So in the following, we take $\epsilon=0.25$ to start our simulation on a domain $[-200,200]\times[-200,200]$. We generated the three crystallites using random perturbations on three small square pathes. The following expression will be used to define the crystallites such as in \cite{yang2017linearly}:
\begin{equation*}
\phi(x_l,y_l)=\overline{\phi}+C\left(\cos(\frac{q}{\sqrt{3}}y_l)\cos(qx_l)-\frac12\cos(\frac{2q}{\sqrt{3}}y_l)\right),
\end{equation*}
where $x_l$, $y_l$ define a local system of cartesian coordinates that is oriented with the
crystallite lattice. The parameters $\overline{\phi}=0.285$, $C=0.446$ and $q=0.66$. The local cartesian system is defined as
\begin{equation*}
\aligned
&x_l(x,y)=xsin\theta+ycos\theta,\\
&y_l(x,y)=-xcos\theta+ysin\theta.
\endaligned
\end{equation*}
we set $512^2$ Fourier modes to discretize the two dimensional space. The centers of three pathes are located at $(150,150)$, $(250,300)$ and $(300,200)$ with $\theta=\pi/4$, $0$ and $-\pi/4$. The length of each square is 40. Figure \ref{fig:fig5} shows the snapshots of the density field $\phi$ at different times. We observe the growth of the crystalline phase. Three different crystal grains grow and become large enough to form grain boundaries finally. We plot the energy dissipative curve in Figure \ref{fig:fig6} using three time steps of $\Delta t=0.01$, $0.1$ and $1$. One can observe that the original energies decrease at all time steps.
\begin{figure}[htp]
\centering
\subfigure[t=0]{
\includegraphics[width=4cm,height=4cm]{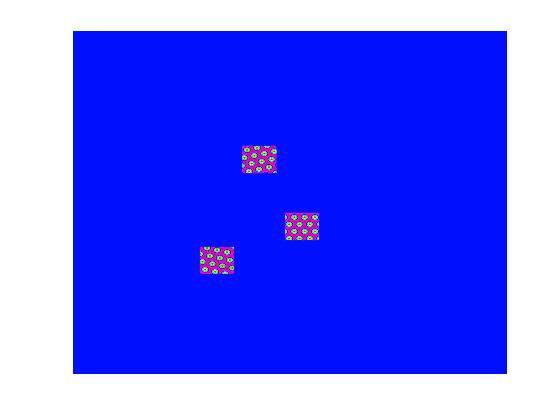}
}
\subfigure[t=200]
{
\includegraphics[width=4cm,height=4cm]{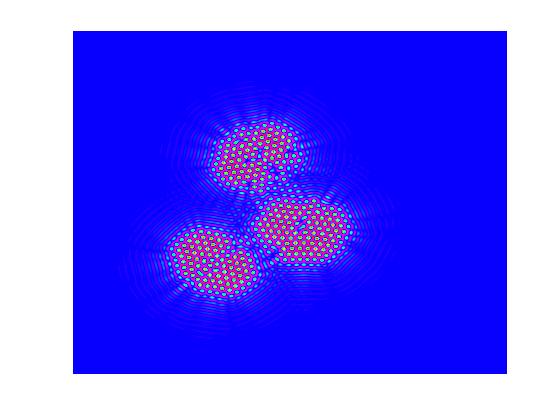}
}
\subfigure[t=250]
{
\includegraphics[width=4cm,height=4cm]{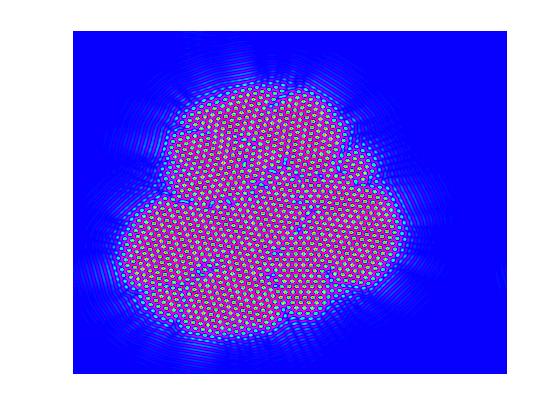}
}
\quad
\subfigure[t=350]{
\includegraphics[width=4cm,height=4cm]{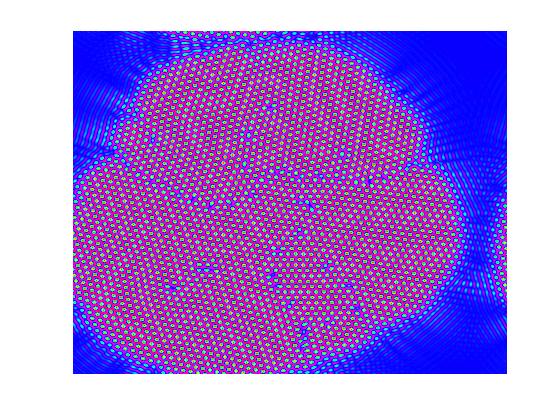}
}
\subfigure[t=400]
{
\includegraphics[width=4cm,height=4cm]{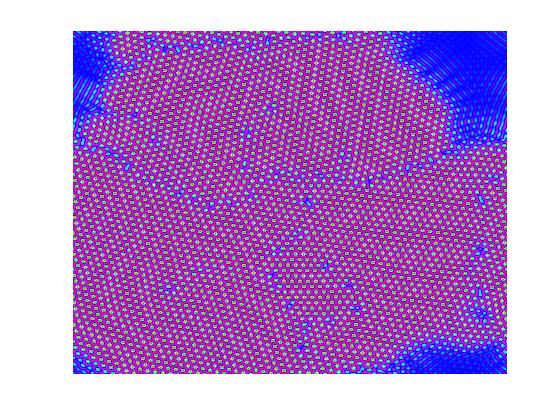}
}
\subfigure[t=500]
{
\includegraphics[width=4cm,height=4cm]{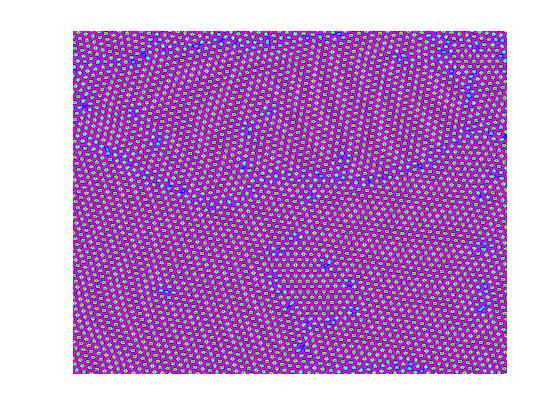}
}
\caption{Snapshots of the phase variable $\phi$ are taken at t=0, 200, 250, 350, 400, 500 for example 4.}\label{fig:fig5}
\end{figure}
\begin{figure}[htp]
\centering
\includegraphics[width=10cm,height=7cm]{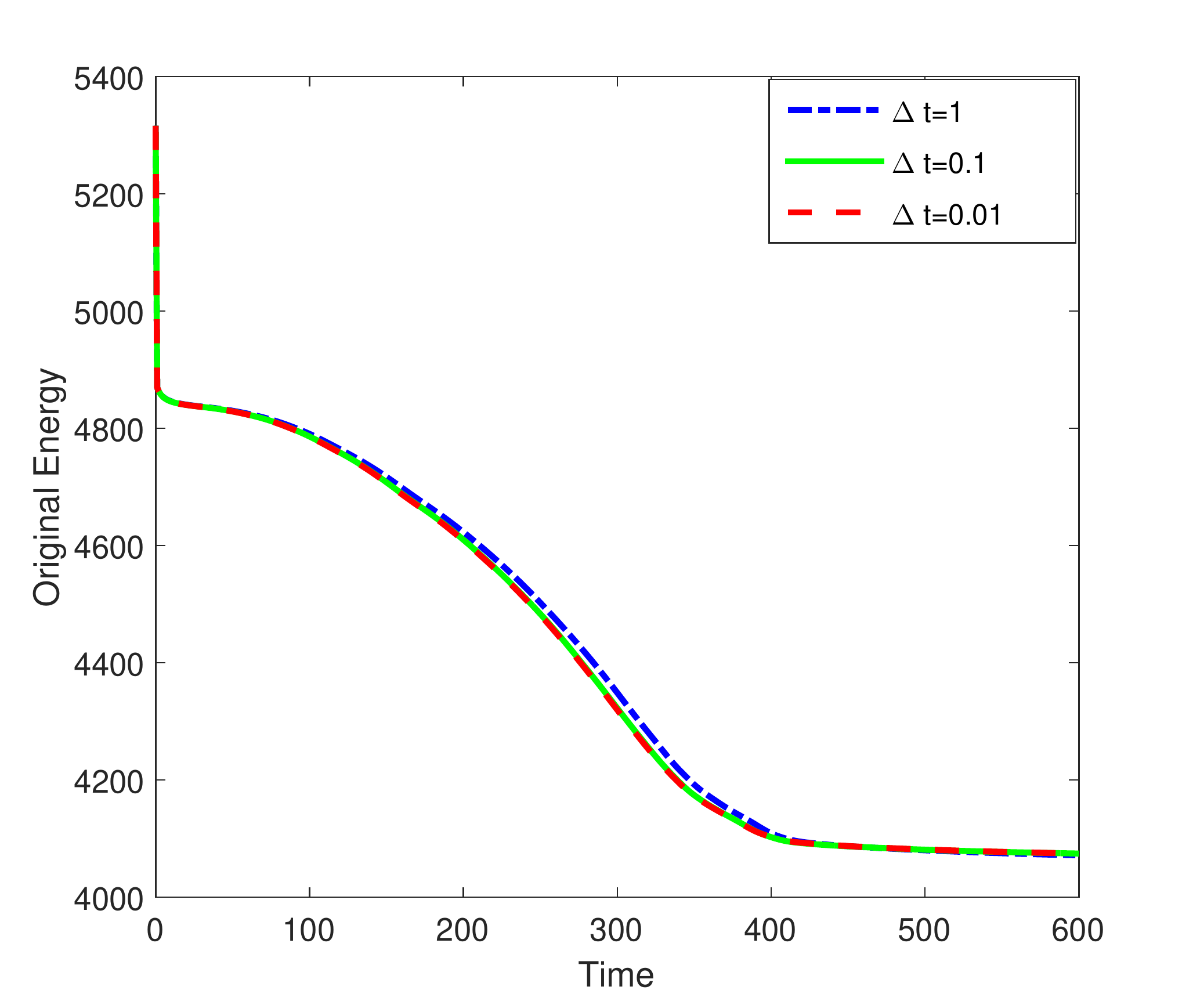}
\caption{Energy evolution of PFC model for example 4 with $\Delta t=0.01$, $0.1$ and $1$.}\label{fig:fig6}
\end{figure}
\section*{Acknowledgement}
No potential conflict of interest was reported by the author. We would like to acknowledge the assistance of volunteers in putting together this example manuscript and supplement.
% \section*{References}
\bibliographystyle{siamplain}
\bibliography{Reference}

\end{document}